\documentclass[12pt,reqno,twoside]{amsart}
\usepackage{amsthm,amsmath,amssymb,mathrsfs,amsbsy}
\usepackage{graphicx,epsfig,wrapfig,sidecap,subfig}
\usepackage[all]{xy}
\usepackage{proof}
\usepackage[usenames]{color}
\usepackage{hyperref}
\usepackage{bm}
\usepackage[capitalize]{cleveref}
\usepackage{stmaryrd}
\usepackage{tikz}
\usepackage{mathdots}
\usepackage{pgf}

\xyoption{dvips}
\xyoption{color}

\usepackage{color}

\title[Quasicrystals, lattices and dense forests]{Cut--and--project
  quasicrystals, lattices, and dense forests}
\date{}
\author{Faustin Adiceam}
\address{University of Manchester, United Kingdom, {\tt faustin.adiceam@manchester.ac.uk}}
\author{Yaar Solomon}
\address{Ben-Gurion University of the Negev, Israel, {\tt yaars@bgu.ac.il}} 
\author{Barak Weiss}
\address{Tel-Aviv University, Israel, {\tt barakw@post.tau.ac.il}}

\newcommand{\N}{{\mathbb{N}}}
\newcommand{\Z}{{\mathbb{Z}}}
\newcommand{\Q}{{\mathbb {Q}}}
\newcommand{\R}{{\mathbb{R}}}

\newcommand{\T}{{\mathbb{T}}}

\newcommand{\sect}{\mathcal{S}}
\newcommand{\window}{K}

\newcommand{\df}{{\, \stackrel{\mathrm{def}}{=}\, }}

\newcommand{\vre}{\varepsilon}
\newcommand{\Vol}{\mathrm{Vol}}
\newcommand{\covol}{\mathrm{coVol}}
\newcommand{\dist}{\mathrm{dist}}

\newcommand{\norm}[1]{\left\|{#1}\right\|}
\newcommand{\inpro}[2]{\left\langle{#1},{#2}\right\rangle}

\newcommand {\ignore}[1]  {}

\newcommand {\equ}[1]{\eqref{#1}}

\font\sb = cmbx8 scaled \magstep0

\long\def\comfaustin#1{\ifdraft{{\color{blue}\sb Faustin says ``#1'' }}\else\ignorespaces\fi}

\theoremstyle{plain}
\newtheorem{thm}{Theorem}[section]
\newtheorem{lem}[thm]{Lemma}

\newtheorem{prop}[thm]{Proposition}
\newtheorem{cor}[thm]{Corollary}

\theoremstyle{definition}
\newtheorem{definition}[thm]{Definition}

\numberwithin{equation}{section}
\swapnumbers

\newif\ifdraft\drafttrue
\usepackage{color}

\begin{document}

\begin{abstract}
Dense forests are discrete subsets of Euclidean space which are
uniformly close to all sufficiently long line segments. The degree of
density of a dense forest is measured by its visibility function. 
We show that cut-and-project quasicrystals are never dense forests,
but their finite unions could be uniformly discrete dense forests. On the
other hand, we show that finite unions of lattices typically are dense forests,
and give a bound on their visibility function, which is close to
optimal. We also construct an explicit finite union of lattices which
is a uniformly discrete dense forest with an explicit bound on its
visibility. 
\end{abstract}

\maketitle


\vspace{4mm}

\begin{flushright}
\emph{\`A la m\'emoire d'\'Evariste Adiceam (1949--2018).}
\end{flushright}

\vspace{4mm}

\section{Introduction}\label{sec:introduction}
A set $Y \subset \R^n$ is called {\em uniformly discrete} if there is
a uniform lower bound on the distance between two distinct points of
$Y$, and {\em of finite density} if 
$$
\limsup_{T \to \infty} \frac{\# (Y \cap B(\bm{0},T))}{T^n} < \infty,
$$
where $B(\bm{x},r)$ denotes the ball of radius $r>0$ around $\bm{x}\in \R^n$,
and the distance
of points in $\R^n$ is measured using some norm (the precise choice of norm
will be immaterial in the results that follow and we will switch
between norms as convenient). Note that a uniformly discrete set is of
finite density but the converse need not hold. We say that
$Y$ is a \emph{dense forest} 
if there exists a 
function $\vre \mapsto v(\varepsilon)$ such that
for every $\varepsilon>0$, every line segment of length
$v(\varepsilon)$ comes $\varepsilon$-close to $Y$. A function
$v :  \R_+ \to \R_+$ for which this condition is satisfied is then
referred to as a \emph{visibility function} of $Y$. By considering a
disjoint collection of cylinders of round base $\vre$ and height $v(\vre)$,
one finds that for a dense forest of finite density, there is a constant $c>0$
such that for all $\vre>0$, 
\begin{equation}\label{eq: trivial bound}
v(\vre) \geq c 
\vre 
^{-(n-1)}.
\end{equation}

A well-known open question of Danzer is whether
there is $Y \subset \R^2$ which is of finite density and for which the
lower bound \equ{eq: trivial bound} is sharp up to the choice of $c$;
i.e.~whether there is a dense forest in the plane of finite density with $v(\vre)
= O\left(\frac{1}{\vre} \right)$.
Interest in Danzer's question has led to some interest in dense
forests of finite density, and uniformly discrete ones, with
visibility functions which are close to the bound given by \equ{eq:
  trivial bound}. We mention four papers which are important
for our discussion. A paper of Bishop \cite{Bishop} gave a
construction, attributed to Peres, of a dense forest of finite density
in $\R^2$. The construction will be reviewed below. The set in
question is a union of three explicit translated lattices, and the bound $v(\vre)
= O \left( 
\vre
^{-4 }\right)$ was
given for the visibility function. The set considered in \cite{Bishop}
is not uniformly discrete, and the first proof of the existence of a 
uniformly discrete dense forest was given in \cite{SW},
with an explicit set in $\R^n$ for any $n\geq 2$, but without an
effective bound on the visibility 
function. In \cite{Adiceam}, for every $n \geq 2$ and every $\eta>0$,
a probabilistic construction was given 
which gives rise to sets of bounded density in $\R^n$ satisfying a visibility
bound $v(\vre) = O\left(
\vre
^{-(2n-2+\eta)} \right).$ Note
that
in the case $n=2$, this improves on \cite{Bishop}, but is
not yet close to \equ{eq: trivial bound}. In \cite{Alon}, Alon gave a
probabilistic argument, which showed the existence of a uniformly discrete dense
forest in
$\R^2$ with a visibility bound $v(\vre) =
O\left(
\vre
^{-(1+o(1))}\right)$. The construction of Alon could
be adapted to higher dimensions as well, and yields sets which come
very close to the lower bound \equ{eq: trivial bound}. However the sets of
\cite{Alon} were not given explicitly. 

Thus it is natural to search for sets $Y \subset \R^n$ with the following
properties:
\begin{itemize}
\item
They are explicitly described.
\item
They are uniformly discrete. 
\item
They are dense forests and their visibility bound comes close to \equ{eq: trivial bound}.
\end{itemize}

The explicit sets we will consider involve periodic and almost
periodic sets, and their finite unions. Note that a lattice is clearly
uniformly discrete but is not a dense forest, and the same holds for a
periodic set (a finite union of translates of one lattice). In
fact a periodic set misses a neighborhood of some affine subspace
of codimension one. On the other hand, as Peres showed, a union of
finitely many periodic sets could be a dense forest. Such a
finite union is 
clearly of finite density, and it is sometimes
uniformly discrete (see Proposition \ref{prop: uniformly discrete unions}). Another
source of explicit constructions are {\em cut-and-project sets} or {\em
  model sets}, which are intensively studied in the literature on
aperiodic structures, see e.g. \cite{BG}. We review their definition
in \S  \ref{subsec: cut and project}. It is well-known that cut-and-project sets are
uniformly discrete. However our first result shows that they cannot be dense forests.

\begin{thm}\label{thm:CP_are_not_DF} Let $Y \subset \R^n$ be a
  cut-and-project set. Then $Y$ is not a dense forest; in fact 
 there exists $\varepsilon>0$ and a $(n-1)$-dimensional affine
 subspace $Z$ of $\R^n$ such that $Y$ contains no points in the
 $\vre$-neighborhood of $Z$. 
\end{thm}

Nevertheless cut-and-project sets can be used to construct interesting
examples of dense forests. Namely we have: 

\begin{thm}\label{thm: immediately implied}
	There exist uniformly discrete dense forests in $\R^2, \,
        $ which
        are a finite union of cut-and-project sets. 
\end{thm}

The uniformly discrete dense forest in
Theorem \ref{thm: immediately implied} is given 
explicitly. However our proof does not provide bounds on its
visibility function. We are able to construct other sets which are
finite unions of translated lattices, for which we have good
visibility bounds. Namely we have:

\begin{thm}\label{thm: finitely many lattices}
There is a union of three translated lattices in $\R^2$ which is a
uniformly discrete dense forest with visibility bound $v(\vre) =
O\left(\vre^{-(5+\eta)}\right)$ for any $\eta>0$. 
\end{thm}

The three translated lattices are given completely explicitly, see
\S \ref{subsec:
  three lattice construction}. 
Removing the condition of uniform discreteness and using more
translated lattices, we are able to get much
better visibility bounds: 

\begin{thm}\label{thm: set with bounds}
For each $n \geq 2$, each $s \geq n$ and each $\eta>0$, for a.e. choice of
$n s $ 
 lattices in $\R^n$, their union is a dense forest with visibility
function satisfying 
$$
v(\vre) = O
\left( 
\vre
^{-(n-1+\alpha_n(s)+\eta)}  \right),
$$
where 
$$\alpha_n(s)  = \frac{n(n-1)^2}{s-(n-1)} \underset{s\to \infty}{\longrightarrow} 0.$$
\end{thm}

The measure implicit in this a.e.~statement is defined by choosing at
random an $s$-tuple $\bm{\Theta}$ of vectors in $\R^{n-1}$ and applying an
explicit construction described in \S \ref{subsec:
  Peres construction}. 
Moreover this a.e.~set
is described by the explicit 
condition of being {\em uniformly Diophantine}, which we
introduce in this paper (see \S \ref{sec: diophantine implies 
  visibility}).  

In the special case $n=2$, our construction is a generalization of the
construction of Peres mentioned above. Recall that Peres used a union
of three explicit translated lattices to obtain a dense forest in $\R^2$, see Figure
\ref{fig: Peres construction}. His
argument used a Diophantine inequality to give a visibility bound of
$O \left( 
\vre
^{-4}\right)$ for this set; our
analysis, applied to the same set, yields a better bound of $O \left(
\vre
^{-3}\right)$ and shows how, by choosing more
lattices whose generators  satisfy a different Diophantine condition,
one can improve further on this bound. \\

%
%
%
%


\begin{center}
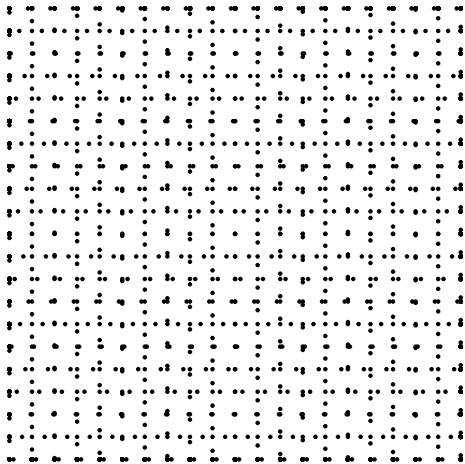


\begin{minipage}{0.4\textwidth}

\begin{tikzpicture}[scale=0.3]
\foreach \x in {-10,...,10}{
\foreach \y in {-10,...,10}{
{\fill (\x,\y) circle (0.1) ;}}}

\foreach \x in {-10,...,10}{
\foreach \y in {-30,...,30}{
{\fill (\x,{max(-10,{min(1.618*\x+\y,10)})}) circle (0.1) ;}}}

\foreach \x in {-10,...,10}{
\foreach \y in {-30,...,30}{
{\fill (-{max(-10,{min(1.618*\x+\y,10)})}, \x) circle (0.1) ;}}}
\end{tikzpicture}
\end{minipage}
\begin{minipage}{0.7\textwidth}
\captionsetup[figure]
{labelformat=simple, margin={-20pt,20pt}}%
\captionof{figure}{Peres' construction: a union of three translated lattices in the plane which is a
  dense forest.\\}
\label{fig: Peres construction}
\end{minipage}

\end{center}


Note that all the examples considered in this paper are known not to be
Danzer sets, i.e.~they cannot realize the bound \equ{eq: trivial
  bound}, see \cite{BW, SW}. 

\subsection*{Organization of the paper} 
After some generalities on cut-and-project constructions and tori, we prove
Theorems \ref{thm:CP_are_not_DF} and \ref{thm: immediately implied}
in \S \ref{sec:CP_are_not_DF} and \S \ref{sec:unions_of_CP_are_DF}
respectively. The proofs rely on viewing cut-and-project sets as
return times to a section in certain higher dimensional toral
flows. In \S \ref{sec: finite unions of lattices} we introduce 
the condition of being uniformly Diophantine, and state
Proposition \ref{propreduc}, which asserts that this condition
ensures a certain uniform rate of equidistribution for
translations on tori. We show that such uniform equidistribution
implies visibility bounds for certain finite unions of translated
lattices, and thus reduce Theorems \ref{thm: finitely many lattices} and 
\ref{thm: set with bounds} to
the verification of the existence of uniformly Diophantine matrices. Proposition \ref{propreduc} is proved in 
\S \ref{sec: toral equidistribution}, and is used to derive Theorem
\ref{thm: finitely many lattices}. In \S\ref{sec: diophantine 
  notions} we develop a `metric theory' related to the property of being uniformly Diophantine,
from which we deduce the existence of the required matrices.
We conclude the paper with some
open problems. 

\subsection*{Acknowledgments} 
The authors are grateful to the anonymous referee for helpful
suggestions. They also thank the referee for producing Figure
\ref{figure:forest} and agreeing to reproduce it in this paper. The authors gratefully acknowledge the support of grants
EP/T021225,  BSF 2016256 and ISF 
2095/15. The first named author wishes to thank Federico Ardila for a
talk given at the University of Waterloo in 2018 which turned out to
be most illuminating to solve some of the questions raised in this
paper.  


\ignore{
\section{Background}\label{sec:background}
 \comfaustin{More pieces of notation to be introduced here after all
   sections are written up.}\\
Throughout the paper we use the following notation: $x_1,\ldots,x_n$
denote the coordinates of a vector $x\in\R^n$, $\norm{x} \df
\sqrt{x_1^2+\ldots,x_n^2}$ the Euclidean norm of $x$, and $B(x,r) =
\{y\in\R^n:\norm{x-y}<r\}$. The standard inner product in $\R^n$ is
denoted by $\inpro{x}{y}$, and for a linear subspace $V$ we denote by
$V^\perp$ its orthogonal complement with respect to that inner
product. For $A,B\subseteq\R^n$, let $\dist(A,B)=\inf\{\norm{a-b}:a\in
A, b\in B\}$. In the case that $A,B\subseteq\T^n$ we use similar
notation: $\dist_{\T^n}(A,B)=\inf\{d_{\T^n}(a,b):a\in A, b\in B\}$,
where $d_{\T^n}(\cdot,\cdot)$ is the standard metric on $\R^n/\Z^n$
given by the minimal distance between cosets representatives. $\Vol(A)$
stands for the Lebesgue measure of a measurable set $A$. The later
notation will also be used for fundamental domains of lattices in
$k$-dimensional subspces of $\R^n$, in which case $\Vol(\cdot)$ refers
to the $k$-dimensional Lebesgue measure on that subspace.    
}

\section{Preliminaries}
In this section we set our notation and collect some results we will
use. 

\subsection{Lattices}\label{subsec: lattices}
A subset $\Lambda \subset \R^d$ is called a {\em lattice} if there are
$\bm{v}_1, 
\ldots, \bm{v}_d$ such that $\Lambda = \mathrm{span}_{\Z} (\bm{v}_i) = \Z
\bm{v}_1 \oplus \cdots \oplus \Z \bm{v}_d$ (note that in this paper
lattices are always of full rank).  
Fix a norm on $\R^d$ and denote by 
$$\lambda_1\left(\Lambda \right)\; \le \; \dots \; \le \;
\lambda_d\left(\Lambda \right)$$ 
the {\em successive minima} of $\Lambda$; that is, 
$\lambda_k(\Lambda)$ is the minimal $r>0$ 
for which $\Lambda$ contains $k$ linearly independent vectors
of norm at most $r$. The successive minima depend on the norm chosen,
and unless otherwise specified, we will use the Euclidean norm. Let
$\mu(\Lambda)$ denote the covering 
radius of $\Lambda$; that is,  
\[
\mu\left( \Lambda\right)\, 
= \; \sup_{ \bm{x}\in\R^d}\; \inf_{\bm{\lambda}\in\Lambda}\; \norm{\bm{x}-\bm{\lambda}}. \label{coveringradius} 
\]
%
 It then follows from Jarn\'ik's Transference Theorem
 (cf.~\cite[Theorem~23.4 p.381]{gruber}) that 
\begin{equation}\label{jarniktransf}
\frac{1}{2}\cdot \lambda_d\left(\Lambda \right) \; \le \; \mu\left(
  \Lambda\right)\; \le \; \frac{d}{2}\cdot \lambda_d\left(\Lambda
\right). 
\end{equation}
When using other norms, the constants $1/2, d/2$ should be replaced in
\eqref{jarniktransf} by other constants depending  on the norm and on
$d$.

A {\em translated lattice} or {\em grid} is a set of the form $\bm{x}
+ \Lambda$ where $\Lambda$ is a lattice, and $Y \subset \R^d$ is
called {\em periodic} if it is of the form $Y = \bigcup_{j=1}^s
\left(\bm{x}_j + \Lambda\right)$ for some lattice $\Lambda$ and $\bm{x}_j \in
\R^d$. 

\subsection{Cut--and--project quasicrystals}\label{subsec: cut and project}
We now define {\em cut-and-project sets}, which are an important
source of aperiodic but ordered discrete sets in mathematical
physics. For background and history we refer the reader to \cite{BG}. 
Let $n, k, N$ be integers with $n \geq 1, k\geq 1$ and $N = n+k$, and write
$\R^N=V_{phys} \oplus V_{int}$ for subspaces satisfying $\dim V_{phys}
=n$, $\dim V_{int} =k$. The spaces $V_{phys}$ and
$V_{int}$ are called the 
\emph{physical} and the \emph{internal spaces} respectively, and we
denote by $\pi_{phys}:\R^N \to V_{phys}$ and $\pi_{int}: \R^N \to V_{int}$
the projections associated with the direct sum decomposition. Let $L
\subset \R^N$ be a translated lattice, and let $W \subset V_{int}$ be
a bounded set. The set

\[ \Lambda(L, W) \df \pi_{phys} \left(L \cap   \pi_{int}^{-1}(W)\right) \]
is called a \emph{cut-and-project set}, $W$ and $L$ are its {\em
  window} and {\em lattice}, and $(n, N)$ are its {\em associated dimensions}.

In the literature there are two slightly different conventions
regarding cut-and-project sets. In the first, one fixes
$V_{phys}$ (resp.~$V_{int}$) to be the space parallel to the first $n$
(resp.~last $k$) coordinate axes and varies the translated lattice $L$, while in
the second, one fixes $L = \Z^N$ and varies the summands of the direct
sum decomposition $\R^N = V_{phys} \oplus V_{int}$. It will be
convenient for us to use both points of view. Also, in the
literature, various different hypotheses are imposed on the
window and on the lattice. For example it is often assumed that $W$ is
compact and equal to the closure of its interior, and that
$\pi_{phys}|_L$ is injective and $\pi_{int}(L)$ is dense. We emphasize
that we do not require these assumptions and only assume that $W$ is bounded.  


\subsection{Tori}
We will use boldface
letters to denote vectors in $\R^N$ and write 
$\bm{x} \cdot \bm{y}$ for the standard inner product of
$\bm{x}, \bm{y}$, $\|\bm{x}\| = \sqrt{\bm{x} \cdot
  \bm{x}}$  and $\bm{x}^{\perp}$ for $\{\bm{y} :
\bm{x} \cdot \bm{y} =0\}$. 
A {\em } torus is the quotient $V/ \Lambda$ for a finite dimensional vector
space $V$ and a lattice $\Lambda \subset
V$. The {\em standard torus} $\R^N/\Z^N$ will be denoted by $\T^N$, and $\pi:
\R^N \to \T^N$ will be the projection. For a subspace $V \subset
\R^N$, the restriction of the standard inner product to $V$ is an
inner product, we can use this inner 
product to induce a volume form on $V$, as well as on the quotient
torus $T = V/\Lambda$.  For a Borel subset $A \subset T$ we denote  by
$\Vol_T(A)$ (or $\Vol(A)$ if confusion is unwarranted) its measure
with respect to this volume form. 

A subspace $V \subset \R^N$ is
called {\em rational} if it is the set 
of solutions of a system of linear equations with rational
coefficients. This happens if and only if $V \cap \Z^N$ is a lattice
in $V$, or equivalently, $\pi(V)$ is a torus in $\T^N$. An {\em affine
  subtorus} of $\T^N$ is a translate $\bm{x} + \pi(V)$ for $V$ a
rational subspace of $\R^N$. Equivalently, it is a coset in the
quotient $\T^N / \pi(V)$. We will always
equip $\R^N$ with the standard inner product and use its restriction
to a rational subspace $V$ to equip $\pi(V)$ with a volume.

Let $V \subset \R^N$ be a rational subspace and $T  = \pi(V) \subset
\T^N$ be the corresponding subtorus, and denote by $\Vol_T$ the
corresponding measure on $T \cong V/( V \cap
\Z^N)$. The quotient space $\T^N/T$ is naturally identified with
$V^{\perp}/\pi^{\perp}(\Z^N)$ where $\pi^{\perp} : \R^N \to V^{\perp}$
is orthogonal projection, and we let 
$\Vol_{\T^N/T}$ be the volume on
$\T^N/T$ obtained by using the standard inner product on 
$V^{\perp}$. With these conventions we
have 
\begin{equation}\label{eq: reciprocal}
\Vol_T(T) \cdot \Vol_{\T^N/T}\left(\T^N/T \right)  = \Vol \left(\T^N \right)=1.  
\end{equation}

For a lattice $\Lambda \subset \R^N$, we define its {\em covolume} to
be $\Vol(\R^N/\Lambda)$, and denote this quantity by
  $\covol(\Lambda)$. 
The {\em dual lattice} of $\Lambda$ is
defined by 
\begin{equation}\label{eq:Dual_lattice}
\Lambda^* = \left\{\bm{x} \in \R^N : \forall \bm{y} \in \Lambda, \, \bm{x}
\cdot \bm{y} \in \Z \right\}, 
\end{equation}
and one has (see e.g. \cite[Chap. 1]{Cassels})
\begin{equation}\label{eq: vol covol}
\covol(\Lambda) \cdot \covol(\Lambda^*) =1.
\end{equation}

\subsection{Unions of translated lattices}
Understanding tori and their subtori is related to understanding
closed (additive) subgroups of $\R^N$. Any closed subgroup $H$ of
$\R^N$ is of the form $H= L + V$, where $V \subset \R^N$ is a vector
subspace, $L$ is discrete, and the orthogonal projection of $L$ onto
$V^{\perp}$ is discrete. Here $V$ is the connected component of the
identity in $H$. Recall that two discrete subgroups $L_1, L_2$
are {\em commensurable} if $L_1 \cap L_2$ is of finite index in both
$L_1$ and $L_2$, or equivalently, $L_1 + L_2$ is discrete. Note that
the connected component of $\overline{L_1 + L_2}$ depends only on the
commensurability class of $L_1, L_2$. 

We will need the following:

\begin{prop}\label{prop: uniformly discrete unions}
Let $L_1, \ldots, L_s$ be lattices in $\R^N$. Then the following are
equivalent: 
\begin{itemize}
\item[(a)] There are
$\bm{x}_i, \ i=1, \ldots , s$ such that $\bigcup_{i=1}^s \left( \bm{x}_i +
L_i \right)$ is uniformly discrete.
\item[(b)]
For each $i, j \in \{1, \ldots, s\}$, $\overline{L_i - L_j} \neq
\R^N$. 
\end{itemize}
\end{prop}

\begin{proof}
We first prove $(a) \implies (b)$. Let $\Lambda_i
= \bm{x}_i + L_i$ for each $i$. Assume by contradiction that
$L_i - L_j$ is dense in $\R^N$. Then $\Lambda_i -\Lambda_j = L_i - L_j
+ \bm{x}_i - \bm{x}_j$ is also
dense, and in particular for any $\vre>0$ there are $\bm{x} \in \Lambda_i, \bm{y} \in
\Lambda_j$ such that $0 < \| \bm{x} - \bm{y}\| < \vre.$ This means
that $\Lambda_i \cup \Lambda_j$ is not uniformly discrete. 

In order to prove $(b) \implies (a)$ we will define $H_{ij} =
\overline{L_i - L_j} \varsubsetneq \R^N$ and show that if $\bm{x}_1,
\ldots, \bm{x}_s$ satisfy 
\begin{equation}\label{eq: ij satisfy}
\bm{x}_i - \bm{x}_j \notin H_{ij} \text{ for all } i,j,
\end{equation}
then $\bigcup_{i=1}^s \left( \bm{x}_i +
L_i\right)$ is uniformly discrete. Note that \equ{eq: ij satisfy} holds for
almost every choice of $\bm{x}_1, \ldots, \bm{x}_s$. Let $\vre>0$ be
smaller than the minimal distance between $\bm{x}_i - \bm{x}_j$ and
$H_{ij}$, and also smaller than the minimal distance between two
distinct points in the same $L_i$. Let $\bm{y}_1, \bm{y}_2$ be two
distinct elements of $\bigcup_{i=1}^s
\Lambda_i$. If $\bm{y}_1, \bm{y}_2$ belong to the same $\Lambda_i$ then they
are at least a distance $\vre$ apart, and otherwise we can write
$\bm{y}_1  = \bm{x}_{i'} + \ell_1, \ \bm{y}_2  =  
\bm{x}_{j'} + \ell_2$ with $i' \neq j'$, $\ell_1 \in L_{i'}$ and $\ell_2 \in
L_{j'}$. Then $\ell_2 - \ell_1 \in H_{i'j'}$ and hence
$$
\|\bm{y}_1 - \bm{y}_2\| \geq \|\bm{x}_{i'} - \bm{x}_{j'} - (\ell_2 -
\ell_1)\| \geq \vre,
$$
as required. 
\end{proof}

\begin{cor}\label{cor: 2 grids not enough}
For any two translated lattices $\Lambda_1, \Lambda_2 \subset \R^N$,
if the union $Y= \Lambda_1 \cup \Lambda_2$ is uniformly
discrete, then $Y$ is not a dense forest; in fact there exists
$\vre>0$ and an $(N-1)$-dimensional affine subspace 
$Z \subset \R^N$ such that $Y$ contains no points
in the $\vre$-neighborhood of $Z$. 

\end{cor}

\begin{proof}
Let $\Lambda_i = L_i +\bm{x}_i, \ i=1,2$, let $H = \overline{L_1 -
  L_2}$, and let $V$ be the connected component of the identity in
$H$. Since $Y$ is uniformly discrete, we have by Proposition
\ref{prop: 
  uniformly discrete unions} that $V \varsubsetneq
\R^N$. Let 
$V_0$ be an $(N-1)$-dimensional subspace of $\R^N$ containing $V$ which is
spanned by elements of $H$, and let $P : \R^N \to V_0^\perp$ be the
orthogonal projection. Then the choice of $V_0$ ensures that $P (H)$
is discrete, and since $L_1 \cup  
L_2 \subset H$, we have that 
$Y \subset
(\bm{x}_1 + H) \cup (\bm{x}_2 +H)$. If we take a point $\bm{z} \in
V_0^\perp \setminus P \left( (\bm{x}_1 + H) \cup (\bm{x}_2 +H) \right)$,
then the conclusion of the Corollary will be valid with $Z =
P^{-1}(\bm{z})$.  
\end{proof}
Note that by Theorem \ref{thm: finitely many lattices}, it is
possible to obtain a uniformly 
discrete dense forest as a union of {\em three} translated lattices in
$\R^2$. 

\subsection{Visit times in a dynamical system}
Cut-and-project sets also arise as the set of \emph{visit times} to a
section of an $\R^n$-action on $\T^N$, $n < N$. Let $V\cong
\R^n$ be a subspace of $\R^N$. Then $V$ acts on $\T^N$ via the linear action 
\[\bm{v} \cdot \bm{x} = \bm{x}+\pi(\bm{v}).\] 
Following the terminology in \cite{HKW}, for $\bm{x}_0\in\T^N$ and
$\sect\subset\T^N$ we define the \emph{$(n,N)$-toral dynamics set} by
the set of `return times' to $\sect$: 

\begin{equation}\label{eq:def-toral_dynamics_set}
Y_{\sect,\bm{x}_0}\df\{\bm{v}\in V: \bm{v}\cdot \bm{x}_0 \in \sect\}.
\end{equation}

An $\R^n$-action is called {\em minimal} if all orbits are dense. 
It is well-known (see e.g. \cite[Chap. 3]{CFS}) that the above
$\R^n$-action is minimal if and only if $V$ is a {\em totally 
irrational subspace}; namely, if and only if it is contained in no
proper rational
subspaces of $\R^N$. Moreover, for any $\bm{x} \in \T^N$, the closure
$\overline{V \cdot \bm{x}}$ is an affine subtorus of $\T^N$ on which the
$V$-action is minimal; that is $\overline{V \cdot \bm{x}} = \bm{x} + \pi(U)$, where
$U$ is the smallest rational subspace containing $V$.

A subset $\sect \subset \T^N$ is called a \emph{linear section} (for
the $V$-action on $\T^N$) if 
$\sect=\pi(\window)$ for a bounded set $\window \subset U$, where $U$
          is a $k$-dimensional affine subspace of $\R^N$ that is
          transverse to $V$, and $\window$ has non-empty interior in $U$. 

We will repeatedly use the following well-known fact (which was
mentioned without proof in \cite{HKW}): 

\begin{prop}\label{prop: candp toral}
If $U, V$ are subspaces of $\R^N$ with $\R^N = U \oplus V$, and  $\sect
\subset \pi(U)$ is a linear section for the 
associated $V$-action on $\T^N$, then for any $\bm{x_0} \in \T^N$, the set $Y_{\sect, \bm{x_0}}$ is
a cut-and-project set for a decomposition in which $V =
V_{phys}$ and $U = V_{int}$. Moreover any cut-and-project set arises
in this way. 

\end{prop}
\begin{proof}
Let 
$$\bm{x}_0\in \T^N, \ V = V_{phys}, \ U = V_{int},\ L =
-\bm{x}_0 + \Z^N, \ \sect = \pi(\window), \ W = -\window.$$
Then 
\[\begin{split}
\bm{v} \in Y_{\sect, \bm{x}_0}  \iff & \bm{v} + \bm{x}_0 \in \sect = \pi(\window) \\
\iff & \exists \bm{w} \in W \text{ such that } \bm{v} + \pi(\bm{x}_0)  = \pi(-\bm{w}) \\
\iff & \exists \bm{z} \in \Z^N, \, \bm{w} \in W \text{ such that } v +
\bm{x}_0  = -\bm{w} + \bm{z}\\ 
\iff & \exists \bm{z'}\in  L, \, \bm{w} \in W \text{ such that } v +
\bm{w} = \bm{z}' \\
\iff & \exists \bm{z}' \in L \text{ such that } \bm{v} = \pi_{phys}(\bm{z}')  , \, \pi_{int}(\bm{z}')  \in W.
\end{split}
\]
\end{proof}




\subsection{More detailed statements}
Via Proposition \ref{prop: candp toral}, we see that Theorem \ref{thm:CP_are_not_DF}
asserts that $(n,N)$ toral dynamics sets arising from linear sections are
not dense forests. 
On the other hand, in \cite{SW}, a uniformly discrete dense forest
was constructed, using visit times to a section in
an $\R^n$-action on a compact 
homogeneous space. We would like to modify this construction and use
dynamics of linear toral flows instead, but as Theorem \ref{thm:CP_are_not_DF} shows,
if we use a linear section the resulting set will not be a dense
forest. To rectify this we will allow a larger class of sets to serve
as the section $\sect$. 

For this
discussion we will specialize to the case $N=3, \, n=2$, so that a section is
one-dimensional. We say that $\sect \subset \T^3$ is a
\emph{piecewise linear 
  unavoidable section} if  
\begin{enumerate}
	\item $\sect$ is a finite union $J_1 \cup \cdots \cup J_\ell$
          where the $J_i$ are disjoint projections under $\pi$ of
          closed
          line segments in $\R^3$ (of finite length). 
	\item $\sect$ intersects every co-dimension 1 sub-torus, that is $\sect
          \cap \pi(\bm{x} + Q) \neq \varnothing$ for every $\bm{x} \in \R^3$ and
          every 
          $2$-dimensional rational subspace $Q\subset \R^3$.  
\end{enumerate} 

The following results will be proved in \S
\ref{sec:unions_of_CP_are_DF}. They immediately imply 
Theorem \ref{thm: immediately implied}. 
\begin{thm}\label{thm: good sections exist}
	Piecewise linear unavoidable sections in $\T^3$ exist. 
\end{thm}

In fact, as we will see, they exist whenever the dimensions $n, N$
satisfy $N=n+1$, but we will not be using this fact.

\begin{thm}\label{thm:UDDF_as_a_union_of_CP} 
	For every piecewise linear unavoidable section $\sect \subset
        \T^3$, every $\bm{x}_0 \in \T^3$, and 
        every $2$-dimensional subspace 
        $V \subset \R^3$ which does not contain rational lines
        and is
        transverse to $\sect$, the set   
	\[Y = Y_{V, \sect, \bm{x}_0} \df \left\{\bm{v}\in V: \bm{v}
        \cdot \bm{x}_0 \in \sect \right\}\]
	is a uniformly discrete dense forest in $V\cong\R^2$. The set
        $Y$ 
is a finite union of 
cut-and-project sets with associated dimensions
          $(2,3)$, with the same
          physical space $V_{phys} =V$.
\end{thm}

See Figure \ref{figure:forest} for an illustration (provided by the
anonymous referee of this paper) of a visit set to a piecewise linear
unavoidable section for the flow induced by a random plane $V$, on
$\mathbb{T}^3$. 
\begin{figure}[ht!]
	\begin{center}
		\input{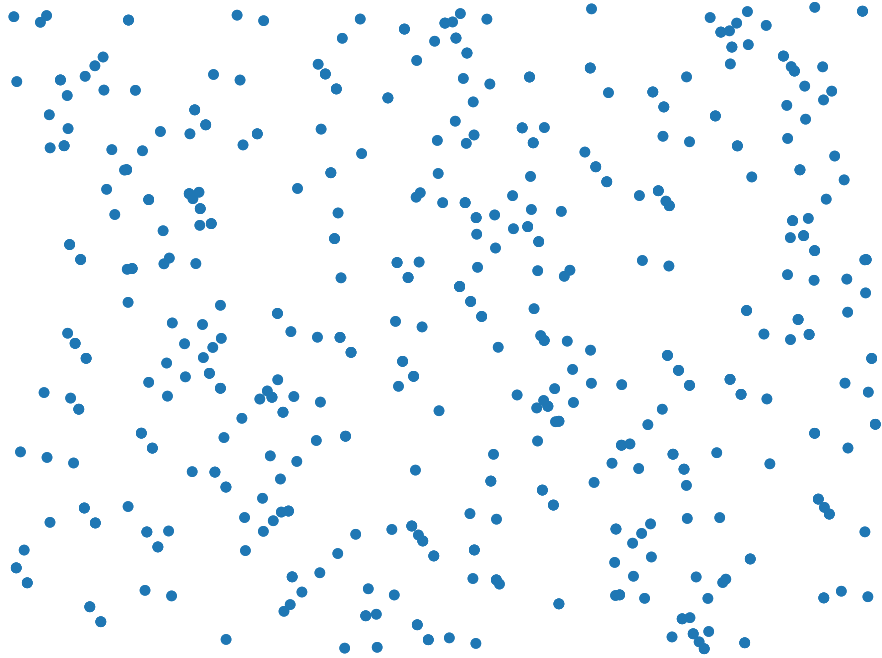}
		\caption{A visit set to a piecewise linear
                  unavoidable section.}
		\label{figure:forest}
	\end{center}
\end{figure} 

\section{Cut-and-project sets are not dense
  forests}\label{sec:CP_are_not_DF}  
In this section we will prove Theorem \ref{thm:CP_are_not_DF}. 
We will need the following well-known fact (see e.g.  \cite[Cor.,
p. 25]{Cassels}):  
\begin{prop}\label{prop: size of q}
	Let $\bm{q} \in \Z^N$ and $Q =\bm{q}^\perp$. Then 
	\begin{equation}\label{eq: 3.1}
	\norm{\bm{q}} \ge \Vol\left(Q / Q \cap \Z^N \right).
	\end{equation}   
	Moreover, if $\bm{q}$ is primitive (i.e. the gcd of its coordinates
        is equal to 1), then we have equality in \eqref{eq: 3.1}.
\end{prop}

\begin{lem}\label{lem:CP_not_DF-main_idea}
	Let $k,N\in\N, 1 \le k<N$, and let $U \subset \R^N$ be a
        $k$-dimensional subspace. Then for every bounded set
        $\window\subset U$ there exists an $(N-1)$-dimensional rational
        subspace $Q\subset\R^N$, and some $\bm{b} \in \R^N$, such that
        $\pi(\window) \cap \pi(Q+\bm{b}) = \varnothing$.    
\end{lem}

\begin{proof}
Given $\window$ and $U$, it suffices to find a rational subspace $Q$ of
dimension $N-1$, and a coset of 
         $\pi(Q)$, $\widetilde{Q} \in \T^N/\pi(Q)$, such that 
	 $\pi(\window) \cap \widetilde{Q} = \varnothing$ (where the
         quotient denotes quotients of abelian
         groups). Indeed, if this happens then any $\bm{b} \in
         \pi^{-1}(\pi(\widetilde{Q}))$ will satisfy the required conclusion. 
	 Given a rational subspace $Q\subset\R^N$ of dimension $N-1$,
         let $P_{Q^\perp}:\R^N\to 
         Q^\perp$ denote the orthogonal projection on $Q^\perp$. Note
         that $\pi(Q)\cong Q/\left(Q\cap\Z^N\right)$ is 
         an $(N-1)$-dimensional sub-torus of $\T^N$, and that the
         space of cosets $\T^N/\pi(Q)$ is parameterized by
         $Q^\perp/P_{Q^\perp}(\Z^N)$. 
	 
If no such coset $\widetilde{Q} \in \T^N/\pi(Q)$ exists then 
$P_{Q^\perp}(\window)$ covers $Q^\perp/P_{Q^\perp}(\Z^N)$, and in
particular $\Vol(P_{Q^\perp}(\window)) \ge
\Vol(Q^\perp/P_{Q^\perp}(\Z^N))$. So it suffices to find a rational
subspace $Q$ with the property that 
	 \begin{equation}\label{eq:the_required_Q}
	 \Vol(P_{Q^\perp}(\window)) < \Vol(Q^\perp/P_{Q^\perp}(\Z^N))
         \stackrel{\equ{eq: reciprocal}}{=}
         \frac{1}{\Vol(Q/Q\cap\Z^N)}\cdotp 
	 \end{equation}

	 Let $\{\bm{u}_1,\ldots,\bm{u}_k\}$ be an orthonormal basis of $U$. By
         replacing $\window$ with a set that contains it, it suffices to
         find $Q$ satisfying (\ref{eq:the_required_Q}) where 
	 \begin{equation}\label{eq:S_is_a_cube}
	 \window = \left\{\sum_{i=1}^k a_i \bm{u}_i : |a_i|\le t \right\},
	 \end{equation}  
	 for some $t>0$. We will look for $Q=\mathrm{span} \, \{\bm{q}\}^\perp$, where
         $\bm{q}\in\Z^N$. By Proposition \ref{prop: size of q},
         it would suffice to find $\bm{q} \in\Z^N$ with
         $\Vol(P_{Q^\perp}(\window)) < \frac{1}{\norm{\bm{q}}}$. For $\window$ as
         in (\ref{eq:S_is_a_cube}), denoting by $\left(\bm{e}_1, \dots, \bm{e}_N\right)$ the standard basis, we have  
	 \[\Vol(P_{Q^\perp}(\window)) \le 2 \sqrt{k} \cdot t \cdot \max
         \left\{ \left| \bm{e}_i \cdot
             {\frac{\bm{q}}{\norm{\bm{q}}}} \right| :
           i\in\{1,\ldots,k\} \right \},\] 
	 and this expression is smaller than $1/\norm{\bm{q}}$ if
	 \begin{equation}\label{eq:Minkowski}
	 \max_{1\le i\le k} \left\{ \left|\bm{e} _i \cdot
             {\bm{q}}\right| \right\} < \delta, \ \ \text{ where }
         \delta \df \frac{1}{2\sqrt{k}t}\cdot 
	 \end{equation} 
Let $L$ be a line in $U^\perp$; then (\ref{eq:Minkowski}) clearly holds if $\bm{q}$ is
         a nonzero integer vector in the $\delta$-neighborhood of $L$; that is, in the set 
	 \[\mathcal{C} \df \left\{\bm{x} \in\R^N :
           \dist(\bm{x} , L) < \delta\right\}. \]
	 Since $\mathcal{C}$ is a convex centrally symmetric body of infinite
         volume, by Minkowski's convex body theorem (see
         e.g. \cite[p. 71]{Cassels}) such an integer vector $\bm{q}$
         exists.  
\end{proof}

\begin{proof}[Proof of Theorem \ref{thm:CP_are_not_DF}]
Replacing $\T^N$ if necessary with $\overline{V \cdot \bm{x}}$, we may assume
that $V$ is totally irrational. If $n = N$ then $Y$ is periodic, that
is $Y$ is a finite union of cosets of a lattice $\Lambda \subset
\R^n$, and
then we may take $Z$ to be parallel to an $(n-1)$-dimensional subspace
spanned by vectors in $\Lambda$. 

So we can assume that $n<N$ and $Y_{\sect,\bm{x}_0}$ is a $(n, N)$-toral dynamics set,
        associated with the action of a totally irrational
        $n$-dimensional space $V$ on the torus $\T^N$. Set
        $k=N-n$. Then, by assumption, $\sect = \pi(\window)$, where $\window$ is a compact subset of a
        $k$-dimensional affine space $U \subset \R^N$. Let $\bm{z}\in U
        \cap V$, and note that $Y_{\sect,\bm{x}_0}-\bm{z} =
        Y_{\pi(\window-\bm{z}),\bm{x}_0}$. Then it suffices to show that
        $Y_{\pi(\window-\bm{z}),\bm{x}_0}$ is not a dense forest. Note that
        $\window-\bm{z} \subset U-\bm{z}$ which is a $k$-dimensional subspace of
        $\R^N$. By Lemma \ref{lem:CP_not_DF-main_idea} there
        exists an $(N-1)$-dimensional rational subspace
        $Q\subset\R^N$, and a coset $\widetilde{Q}$ of it, such that
        $\pi(\window - \bm{z}) \cap \pi(\widetilde{Q}) = \varnothing$. Note
        that $Z \df V \cap \widetilde{Q}$ is a
        $(n-1)$-dimensional affine subspace of $V$ and of
        $\widetilde{Q}$. Since $\pi(\window - \bm{z}) \cap \pi(\widetilde{Q})
        = \varnothing$, and the sets $\pi(\window - \bm{z})$ and
        $\pi(\widetilde{Q})$ are closed in $\T^N$, there exists some
        $\varepsilon>0$ such that $d_{\T^N}(\pi(\window -
        \bm{z}),\pi(\widetilde{Q})) > \varepsilon$. Then we also have
        $d(\pi(\window - \bm{z}),\pi(Z)) > \varepsilon$. That is, any point
        which is $\vre$-close to $Z$ misses $Y_{\sect, \bm{x}_0}-\bm{z}$, and
        this proves the assertion.  
\end{proof}


\section{An explicit uniformly discrete dense forest using toral flows}
\label{sec:unions_of_CP_are_DF}
In this section we will prove Theorems \ref{thm: good sections exist}
and 
\ref{thm:UDDF_as_a_union_of_CP}. 

In order to see that a piecewise linearly unavoidable section exists,
we refer the reader to Figure \ref{figure:Section_example}. 

\begin{figure}[h!]
	\begin{center}
		\includegraphics[scale=0.9]{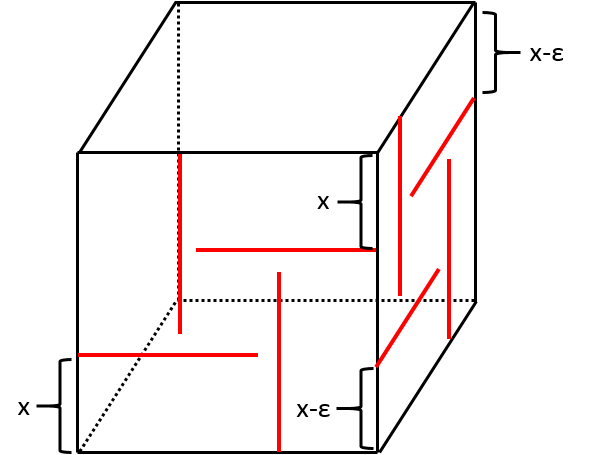}
		\caption{A piecewise linearly unavoidable section in
                  $\T^3$: any 2-torus intersects at least one of the 
                  faces of the cube in a loop, and thus intersects
                  the section. }
		\label{figure:Section_example}
	\end{center}
\end{figure} 

We also sketch an alternative existence proof for linearly unavoidable sections,
which uses toral dynamics and can be generalized to higher dimensions,
i.e. to the case $N = n+1$ for arbitrary $n \geq 2$. 

\begin{proof}[Sketch of another proof of Theorem \ref{thm: good
      sections exist}]
Let $\{T_j : j \in \N\}$ be an enumeration of the 2-dimensional affine
tori passing through the origin in $\T^3$. From 
Proposition \ref{prop: size of q} we have $\Vol(T_j) \to \infty$ and
from \equ{eq: reciprocal}, $\Vol\left(\T^N/T_j \right)\to 0$. 
Let
 $J_1, J_2 ,J_3\subset\R^3$ be closed line segments, in
linearly independent directions, such that the images in  $\sect=\bigcup \pi(J_i)$ are
disjoint. Since the directions are linearly independent, there is a
uniform lower bound on the angle that each $T_j$ makes with at least
one of the $J_i$. Therefore when projecting onto $\T^3/T_j$, for all large
enough $j$, at least one of the $J_i$ projects onto the quotient
$\T^3/T_j$. Thus for
all sufficiently large $j$, and every coset $T'_j$ of $T_j$, we have
$T'_j\cap \sect \neq \varnothing$.

Now by
adding finitely many line segments to $\sect$, and keeping the
property $\pi(J_{i_1}) \cap \pi(J_{i_2}) = \varnothing$, we obtain a
piecewise linear unavoidable section.
\end{proof}

\begin{proof}[Proof of Theorem \ref{thm:UDDF_as_a_union_of_CP}]
This is very close to the argument of  
\cite[Proof of Thm. 1.3]{SW}. First note that uniform discreteness of $Y$ follows
from the fact that the segments comprising 
$\sect$ are closed and disjoint, and the transversality assumption. We prove that $Y$ is a
dense forest by contradiction. If not, then there
exists some $\varepsilon>0$, unit vectors $\bm{w}_j \in V$, and
$L_j \to \infty$ such that the line segments $\ell_j \df \{
\bm{x}_j+t\cdot \bm{w}_j : t \in [0,L_j] \}$ satisfy 
$\dist(Y,\ell_j)\ge\varepsilon$ for all $j$. Denote 
\[K_j \df \left\{\bm{v}\in V: \dist(\bm{v},\ell_j) \le \frac{\varepsilon}{3} \right\},\]
and define a sequence of Borel probability measures $\mu_j$ on $\T^3$ by
\[\forall f \in C(\T^3),\quad  \int_{\T^3} f d\mu_j \df
\frac{1}{\Vol(K_j)} \int_{K_j} f(\bm{v} \cdot \bm{x}_0) dv,\]  
where the integral on the right hand side is with respect to the
Euclidean volume on $V$. By passing to a subsequence we may assume
that $\bm{w}_j\xrightarrow{j\to\infty} \bm{w}$, and that $\mu_j
\xrightarrow{weak-*} \mu$. Since $\bm{w}_j$ is the direction of the long
axis of the cylinder $K_j$, and since the stabilizer of a measure is a
group, it follows that the measure $\mu$ is
invariant under $H \df \mathrm{span} (\bm{w})$. 
By \cite[Chap. 3]{CFS}, every Borel
probability measure on $\T^3$, invariant and ergodic under $H$, is the
Haar measure on some rational 
torus $T\subset\T^3$. Note that such a $T$ cannot be
$1$-dimensional. Indeed if such a $T$ is $1$-dimensional then $T = \pi(H)$,
hence $\bm{w}$ is a rational direction in the physical space $V$,
contradicting the assumption that $V$ does not contain rational
lines. 

Let $g\in C(\T^3)$ be a bump function that is positive on $\sect$
and supported on 
the $\varepsilon / 3$
neighborhood of $\sect$. Recall that since $\sect$ is a piecewise
linear unavoidable section, $\sect$ intersects every
$2$-dimensional rational sub-torus of $\T^3$, and in particular
$\sect \cap T \neq \varnothing$, for every $T$ as above. This implies that $\int_{\T^3}g d\nu
> 0$ for any ergodic $H$-invariant measure $\nu$, and hence by ergodic
decomposition, $\int_{\T^3}g d\mu
> 0$. On the other hand, for every $j$ and for every $v \in K_j$, by
definition $\dist(v, Y)\ge \frac{2\varepsilon}{3}$, thus
$\dist_{\T^3}(\bm{v} \cdot \bm{x}_0,\sect) \ge \frac{2\varepsilon}{3}$ and
$\bm{v} \cdot \bm{x}_0$ misses the support of $g$. This implies $g(\bm{v} \cdot \bm{x}_0)
= 0$, and hence 
$\int_{\T^3} g d\mu_j = 0$ for every $j$, a contradiction to $\mu_j
\xrightarrow{weak-*} \mu$. 

Let $I_1,\ldots,I_\ell$ be 
     line segments whose projections define $\sect$. Then each $\pi(I_i)$ is a
    linear section and hence the set $Y =
    \bigcup_{j=1}^{\ell} Y_{\pi(I_i), \bm{x}_0}$ is a finite union of
    cut-and-project sets as required. 
\end{proof}


\section{Finite unions of translated 
  lattices and uniformly Diophantine
  sets of vectors}\label{sec:peres_construction}\label{sec: finite unions of
  lattices} 
We now move to results concerning finite unions of translated lattices. 

\subsection{More notation.} The following notation will be used
in the rest of the paper. 
Given two expressions $X$ and $Y$, we will use both of the 
notations $X\ll Y$  and $X = O(Y)$ to mean that $X$ and $Y$ are
depending on some variables and there exists a
constant $c>0$ (called the {\em implicit constant}), independent of
these parameters, such that 
$X\le cY$. 

\begin{itemize}
\item
Throughout we will have two dimensions $n$ and $d$ linked by the
relation
$$
n=d+1 \geq 2.
$$
\item
The coordinates of $\bm{x} \in \R^n$ will be denoted by $(x_1, \ldots,
x_n)$. 
\item $\|\cdot \|_{\infty}$ and $\|\cdot \|_{2}$ will denote respectively the
  sup-norm and Euclidean norm, $B_{\infty}(\bm{x}, r)$ and
  $B_2(\bm{x}, r)$ will denote the respective open  balls. When
  making a statement which does not depend on the choice of norm we
  will simply write $\| \cdot \|$ and $B(\bm{x}, r)$ unless these notations are specifically defined otherwise.
\item   $\bm{x\cdot y}$ will denote the usual scalar product between
  the vectors $\bm{x}, \, \bm{y} \in \R^n$.
\item $\pi : \R^n \to \T^n = \R^n/\Z^n$ is the natural projection. 
\item  $\langle\bm{x}\rangle_{\Z^n}$ will denote the distance from
  $\bm{x}\in\R^n$ to $\Z^n$ with respect to the sup-norm. Thus
  $d(\pi(\bm{x}), \pi(\bm{y})) = \langle 
  \bm{x} - \bm{y} \rangle_{\Z^n}$ is the metric on $\T^n$ induced by
  $\| \cdot \|_\infty$.  For $n=1$ we
  will abbreviate this as $\langle x \rangle_{\Z} = \langle x
  \rangle$. 


\item
A real $n\times m$ matrix $A$ (where $n,m\ge 1$) will be
identified with a vector in $\R^{n\times m}$ by concatening its
successive columns. Its transpose will be the $m\times n$ matrix
denoted by $A^T$. 
\end{itemize}

\subsection{On Peres' construction of dense forests}\label{subsec:
  Peres construction}
We recap Peres' explicit construction of a discrete forest of bounded
density, given in \cite{Bishop}. 
Let  $\varphi = (1+\sqrt{5})/2$ be the golden ratio and
let 
$$\mathfrak{F}_1(\varphi)\df \Z^2 \cup \begin{pmatrix} 1 & 0 \\
  \varphi & 1\end{pmatrix}\cdot \Z^2.$$ Thus,
$\mathfrak{F}_1(\varphi)$ is the union of the standard integer lattice
in $\R^2$ with an irrational shear of it\footnote{In fact, in
  \cite{Bishop}, the slightly different set $\left( \left(\begin{matrix} 1/2 \\ 0 \end{matrix}\right) + \Z^2   \right) \cup \left(\begin{matrix} 1 &
    0 \\ \varphi & 1\end{matrix} \right) \cdot \Z^2$ was used in place of $\mathfrak{F}_1(\varphi)$.}.  

 Applying Dirichlet's theorem in Diophantine approximation,
 Peres proved that $\mathfrak{F}_1(\varphi)$ is a dense forest when
 restricting to line segments with slope bounded in
 absolute value by 1 (that is, to those  line  segments ``close to
 horizontal''). His argument ensured a visibility function of
 $O\left(\varepsilon^{-4}\right)$.  
Set  $$\mathfrak{F}_2(\varphi)\df \Z^2 \cup \begin{pmatrix}  \varphi &
  1 \\ 1 & 0\end{pmatrix}\cdot \Z^2$$ and note that
$\mathfrak{F}_2(\varphi)$ is obtained by permuting the role of the
coordinate axes in the definition of $\mathfrak{F}_1(\varphi)$.  This
implies a similar bound for line segments with slope bigger than 1 in absolute
value (that is, any line segment ``close to vertical''). 
%
Thus, defining $$\mathfrak{F}(\varphi)\df
\mathfrak{F}_1(\varphi)\cup\mathfrak{F}_2(\varphi)$$ (which is the
union of three lattices), we have a dense forest with
visibility  function satisfying $v(\vre) = O\left(\varepsilon^{-4}\right).$ 
 See Figures~\ref{foretasyfigure1} and \ref{foretasyfigure}, which represent respectively the  sets of points $\mathfrak{F}_1(\varphi)$ and $\mathfrak{F}_2(\varphi)$. Their union is the dense forest $\mathfrak{F}(\varphi)$ depicted in Figure~\ref{fig: Peres construction}. 

\begin{minipage}{0.4\textwidth}

\begin{tikzpicture}[scale=0.275]
\foreach \x in {-10,...,10}{
\foreach \y in {-10,...,10}{
{\fill (\x,\y) circle (0.1) ;}}}

\foreach \x in {-10,...,10}{
\foreach \y in {-30,...,30}{
{\fill (\x,{max(-10,{min(1.618*\x+\y,10)})}) circle (0.1) ;}}}
\end{tikzpicture}

\captionsetup{margin={0pt,0\textwidth}}%

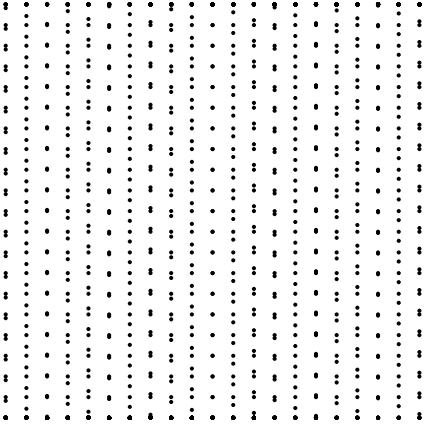
\captionof{figure}{\footnotesize  The $\vre$-thickening of the set $\mathfrak{F}_1(\varphi)$ represented above intersects line segments ``close to horizontal''.}
\label{foretasyfigure1}
\end{minipage}
\hspace{6ex}
\begin{minipage}{0.4\textwidth}
\vspace{8ex}
\begin{tikzpicture}[scale=0.275]
\foreach \x in {-10,...,10}{
\foreach \y in {-10,...,10}{
{\fill (\x,\y) circle (0.1) ;}}}

\foreach \x in {-10,...,10}{
\foreach \y in {-30,...,30}{
{\fill (-{max(-10,{min(1.618*\x+\y,10)})}, \x) circle (0.1) ;}}}
\end{tikzpicture}
\captionsetup{margin={0pt,0\textwidth}}%

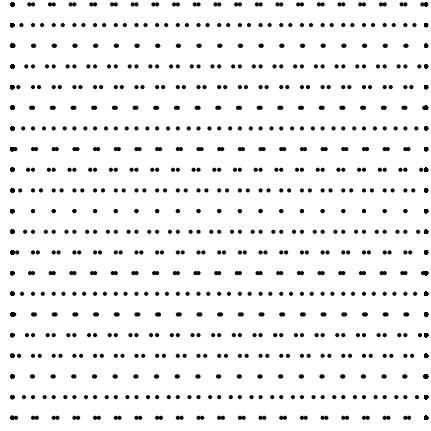
\captionof{figure}{\footnotesize  The $\vre$-thickening of the set
$\mathfrak{F}_2(\varphi)$  represented above intersects line segments ``close to
 vertical''. }\vspace{8ex}
\label{foretasyfigure}
\end{minipage}

%
%
%
%
%
%

\vspace{2ex}

The goal of this section is to generalize Peres' construction,
obtaining dense forests in any dimension which are almost fully
explicit (see Section~\ref{mainresultsec5} for details) and with
good visibility bounds. In particular we will 
improve the visibility bound in Peres' original planar forest.  

Let $J : \R^n \to \R^n$ be the linear transformation that acts by
permutating coordinates as follows:
\begin{equation}\label{eq: permutation matrix}
J \left(x_1, x_2, \ldots, x_{n-1},
    x_n\right)^T = \left(x_2, x_3, \ldots, x_n, x_1\right)^T.
\end{equation}

Given an integer $s\ge 2$, denote by
\begin{equation}\label{defthetas}
\bm{\Theta}_{s,d} = \left(\bm{\theta}_{1}, \, \dots
  \, , \bm{\theta}_{s}\right) 
\end{equation} 
an $s$-tuple of $d$-dimensional vectors.

Then define 
\begin{equation}\label{deff_1}
\mathfrak{F}_1\left(\bm{\Theta}_{s,d}\right)\, \df \,
\bigcup_{i=1}^{s} \begin{pmatrix}  1 & \bm{0}^T\\ \bm{\theta}_{i}  &
  I_{d}\end{pmatrix}\cdot \Z^n,  
\end{equation} 
where $I_{d}$ stands for the $d\times d$ identity matrix.
For $\ell=1, \dots, n$, let 
$\mathfrak{F}_\ell\left(\bm{\Theta}_{s,d}\right)$ denote the image of
$\mathfrak{F}_1\left(\bm{\Theta}_{s,d}\right)$ under $J^{\ell-1}$, i.e.  
\begin{equation}\label{defsubforest}
\mathfrak{F}_\ell \left(\bm{\Theta}_{s,d}\right) = J^{\ell-1} \left(
  \mathfrak{F}_1\left(\bm{\Theta}_{s,d}\right) \right),
\end{equation}
and 
let 
\begin{equation}\label{defforestfinal}
\mathfrak{F}\left(\bm{\Theta}_{s,d}\right) \df
\bigcup_{\ell=1}^{n}\mathfrak{F}_\ell\left(\bm{\Theta}_{s,d}\right). 
\end{equation}
Note that $\mathfrak{F}\left(\bm{\Theta}_{s,d}\right)$ is the union of
at most $ns$ lattices, and that Peres' construction is
$\mathfrak{F}\left(\bm{\Theta}_{2,1} \right)$ with $\bm{\Theta}_{2,1} =\left( 0, \varphi \right)$. 

\subsection{Visibility bounds for these
  forests.}\label{mainresultsec5}\label{sec: diophantine implies 
  visibility}
Recall that a vector $\bm{\theta}\in\R^d$ is said to be Diophantine of
type $\tau>0$ if there exists a constant $c(\bm{\theta})=c>0$ such
that $$\forall \bm{u}\in\Z^d\backslash\{\bm{0}\}, \quad \langle
\bm{u\cdot \theta} \rangle \ge
\frac{c}{\left\|\bm{u}\right\|^{\tau}}\cdotp$$ 

A multidimensional version of Dirichlet's theorem
(see~\cite[Theorem~VI, p.13]{Casselsbis}) implies that
necessarily $\tau\ge d$. The visibility bounds in the
forest~\eqref{defforestfinal} will depend on a strenghtening of this
concept: 

\begin{definition}\label{defudt}
Let $\Phi$ be a non-increasing function tending to zero at
infinity. An $s$-tuple of $d$ dimensional vectors
$\bm{\Theta}_{s,d}$, as in \eqref{defthetas}, is \emph{uniformly
  Diophantine of type $\Phi$} if for any $T\ge 1$ and any
$\bm{\xi}\in\R^d$, there exists $i\in\{ 1, \ldots, s\}$ such that
for all 
$\bm{u}\in\Z^d\backslash\{\bm{0}\}$ with sup-norm at most $T$, 
\begin{equation}\label{lowboundudt}
\left\langle \bm{u \cdot } \left(
    \bm{\xi}-\bm{\theta_i}\right)\right\rangle\; \ge \; \Phi(T). 
\end{equation}
 The set of 
 $\bm{\Theta}_{s,d}$ that are uniformly Diophantine of type $\Phi$
 will be denoted by $UDT_s^{d}(\Phi)$. Thus, $\bm{\Theta}_{s,d}\in
 UDT_s^{d}(\Phi)$ means that  
 $$\inf_{T\ge 1}\; \inf_{\bm{\xi}\in\R^d}\; \max_{1\le i \le s}\;
 \min_{\underset{\bm{u}\in\Z^d}{1\le \left\| \bm{u}
     \right\|_{\infty}\le T}}\left\{ \Phi(T)^{-1}
   \left\langle\bm{u \cdot }  \left(
       \bm{\xi}-\bm{\theta_i}\right)\right\rangle\right\}
 \; \ge \; 1.$$ 
 Also, given $\tau>0$, set 
 $$UDT_s^{d}(\tau)\; \df\; \bigcup_{c>0}UDT_s^d\left(x\mapsto
   cx^{-\tau} \right).$$ 
\end{definition}

It is easily seen that 
the set  $ UDT_s^{d}(\Phi)$ is translation invariant;
that is, for any $\bm{\alpha}\in\R^d$, 
$$ \left(\bm{\theta}_{1},  \,
  \dots \, , \bm{\theta}_{s}\right)\in UDT_s^{d}(\Phi) \ \iff \ 
\left(\bm{\theta}_{1}+\bm{\alpha}, \, \dots \, ,
  \bm{\theta}_{s}+\bm{\alpha}\right)\in UDT_s^{d}(\Phi).$$
In particular, from any uniformly Diophantine set of vectors of a given 
type, one can obtain another uniformly Diophantine  set of vectors of the same type such that one of the latter vectors takes any
predefined value. Also, if $UDT^d_s (\Phi)\neq \varnothing$, then
taking 
 $\bm{\xi}=\bm{0}$ in \eqref{lowboundudt} and using Dirichlet's
theorem, one sees that 
necessarily 
\begin{equation}\label{grothphi}
\Phi\left(T\right) \; = \; O\left(T^{-d} \right). 
\end{equation}

\begin{thm}\label{thmforestperes}
Assume that 
$\bm{\Theta}_{s,d} \in UDT_s^d(\Phi)$. Then the
set $\mathfrak{F}\left(\bm{\Theta}_{s,d}\right)$ constructed
in~\eqref{defforestfinal} is a dense forest in $\R^n$ with visibility
function sa\-tis\-fy\-ing 
\begin{equation}\label{boundvisibility}
v(\vre) = O\left( \left(\varepsilon^{d-1}\cdot
    \Phi\left(d\varepsilon^{-1}\right)^{-1}\right)^d\right). 
\end{equation}
\end{thm}

Theorem~\ref{thmforestperes} will be established in
\S\ref{redproof}. 
Note that as the bound on the uniformly Diophantine type
comes closer to the upper bound \eqref{grothphi}, the bound \eqref{boundvisibility}
on the visibility approaches the optimal \equ{eq: trivial bound}. 

A number $\theta \in \R$ is {\em badly approximable} if it is of Diophantine type
$\tau =1$. It is well-known that the golden ratio $\varphi$ is badly
approximable. 
It will be shown in \S\ref{udtmultilin} that any 
$\left(\alpha, \beta\right)^T\in\R^2$, where $\beta-\alpha$ is a
badly approximable number, belongs to the set
$UDT_2^1(3)$. Combined with Theorem \ref{thmforestperes}, this implies
that the visibility
bound in Peres' original forest can be improved from
$O\left(\varepsilon^{-4} \right)$ to $O\left( \varepsilon^{-3}\right)$. 

The property of being a uniformly Diophantine set of vectors will 
be related in \S\ref{udtmultilin} to an explicit Diophantine
condition. As a consequence, the existence of such sets will be
guaranteed in any dimension. More precisely, the following result will
be established in \S\ref{proofthm5.3}: 

\begin{thm}\label{thmforestperesbis}
Assume that $s\ge d+1$. 
Let $\Phi$ be a non-increasing function tending to zero at
infinity such that  
\begin{equation}\label{growthconstraint}
\liminf_{T\rightarrow \infty}\frac{\Phi(2T)}{\Phi(T)}>0
\end{equation}
and 
\begin{equation}\label{convergencecondi}
\sum_{m=1}^{\infty}2^{m d (s+1)}\Phi\left(2^m\right)^{s-d}\; < \; \infty.
\end{equation}
Then, with respect to the $d\times s$-dimensional Lebesgue measure,
for almost all $\bm{\Theta}_{s,d}$ there is
$c=c\left( \bm{\Theta}_{s,d}\right)>0$ such that $\bm{\Theta}_{s,d}
\in UDT_s^d\left(c\Phi\right)$. 
\end{thm}

As an immediate consequence of Theorems \ref{thmforestperes} and \ref{thmforestperesbis} we obtain:

\begin{cor}\label{coroforest}
Under the assumptions of Theorem~\ref{thmforestperesbis}, the
visibility in the dense forest
$\mathfrak{F}\left(\bm{\Theta}_{s,d}\right)$ constructed
in~\eqref{defforestfinal} can be bounded by~\eqref{boundvisibility}
for almost all $\bm{\Theta}_{s,d}$. 
\end{cor}


For instance, by setting 
$$\Phi(T)=
T^{-\left(\frac{d(s+1)}{s-d} +\eta\right)} \ \ \text{ for } \eta > 0, $$
where $s\ge d+1$, one sees that Theorem \ref{thm: set with bounds} is
a consequence of Corollary \ref{coroforest}. Additional improvements
are possible by setting 
$$\Phi(T)=
T^{-\left(\frac{d(s+1)}{s-d}\right)} \, \log(T)^{-\beta}$$
for appropriately chosen $\beta = \beta_{s,d}$.

\ignore{

the visibility bound of \eqref{boundvisibility}

reads 
\begin{equation}\label{exponents}
O\left(\vre^{-(d+\alpha_d(s) + d\eta)} \right), \ \ \text{
  where }   
\alpha_d(s)\, = \, \frac{d^2(d+1)}{s-d} 
\end{equation}
In the planar case (that is, when $n=d+1=2$), this specializes to the
visibility bound $$O\left(\vre^{-1-2/(s-1)}\cdot
  \left(\log\vre^{-1}  \right)^{1/(s-1)+\eta} \right).$$

By letting the parameter $s$ tend to infinity in~\eqref{exponents},
Corollary~\ref{coroforest} thus provides a construction of a dense
forest in any dimension $n=d+1\ge 2$ with visibility arbitrarily close
(in a suitable sense) to the best possible \comfaustin{Justify this
  claim more carefully here or in the introduction} bound
$O\left(\vre^{-d}\right)$ that can be expected. The construction
is furthermore explicit up to the probabilistic choice of a matrix
$\bm{\Theta}_{s,d}$ which,  as shown  in Section~\ref{udtmultilin},
can nevertheless be chosen from a set of matrices satisfying a fully
explicit Diophantine condition. 

}
\subsection{Reduction to a Diophantine
  statement}\label{redproof} 
In this subsection we will examine what it means for a line segment
$\mathcal{L}$ to be $\vre$-close to the set $\mathfrak{F}
\left(\bm{\Theta}_{s,d}\right)$ defined 
in \eqref{defforestfinal}. For the computations in this subsection it
will be most convenient to work with the sup-norm on $\R^n$, and so in
this section $\|\bm{x}\| = \|\bm{x}\|_\infty$. As all norms on $\R^n$
are bi-Lipschitz equivalent to each other, and the problems we
consider are insensitve to multiplications by constants depending on
dimension, this involves no loss of generality.

Let 
$$0< \vre<\frac12$$ 
and let  $\mathcal{L}$ be the parameterized  line
segment  
\[ \mathcal{L} \df \left \{\bm{\alpha} t +\bm{\beta}: t \in [0, M] \right \}, \]
where $\bm{\alpha},
\bm{\beta} \in \R^n$, 
$\|\bm{\alpha}\| =1,$ so that $ M\ge 0$ is at most the length of
$\mathcal{L}$ and at least a fixed constant multiple of it. 

By \eqref{defsubforest} and \eqref{defforestfinal},
$\mathcal{L}$ is within distance $\vre$ of 
$\mathfrak{F} \left(\bm{\Theta}_{s,d}\right)$ if and only if for
some $\ell$, $J^{\ell}(\mathcal{L})$ is within distance $\vre$ of
$\mathfrak{F}_{1}\left(\bm{\Theta}_{s,d}\right)$. 
Since the 
matrix  $J$ in \equ{eq: permutation matrix} permutes the coordinates, there
is no loss of generality in assuming that $\|\bm{\alpha}\| =
|\alpha_1|$ (where $\alpha_1$ denotes the first coordinate of $\bm{\alpha}$). Also by
switching endpoints of $\mathcal{L}$ if necessary we can assume
$\alpha_1 =1$. Thus we now assume 
\begin{equation}\label{normalisationalpha}
\alpha_1 = \|\bm{\alpha}\| =  1,
\end{equation} 
and study when $\mathcal{L}$ comes $\vre$-close to the set 
$\mathfrak{F}_{1}\left(\bm{\Theta}_{s,d}\right)$ defined by  \eqref{deff_1}. 

Given $k\in\Z$, and using \eqref{normalisationalpha}, we see that
$\mathcal{L}$ intersects the hyperplane $\left\{\bm{x}: x_1=k \right\}
$ precisely when $\beta_1\le k \le \beta_1+M$, and the intersection
point is given by 
$\bm{\alpha}\left(k-\beta_1\right)+\bm{\beta}.$
It follows from \eqref{deff_1} that this point comes
$\vre$-close to $\mathfrak{F}_{1}\left(\bm{\Theta}_{s,d}\right)$ when
there exists an index $1 \le i\le s$ such that  
$$ \left\langle \bm{\alpha}(k-\beta_1)+\bm{\beta} - k
  \bm{\theta}_i\right\rangle_{\Z^d}\; = \; \left\langle k
  \left(\bm{\alpha}- \bm{\theta}_i\right)+\bm{\beta}-\beta_1\bm{\alpha}
\right\rangle_{\Z^d}\; < \; \vre .$$  

Write $k=\lceil \beta_1
\rceil+m$, where $0 \le m\le  M$ is an integer and where $\lceil \,
\cdot \, \rceil$ denotes the ceiling function. Then the preceding
discussion shows:

\begin{prop}\label{prop: have shown}
Suppose that 
\begin{equation}\label{distfibregeneralised}
\forall \bm{\xi}, \bm{\zeta}\in \R^d,
  \ \exists 0\le m\le M, \ \exists i\in\{1, \ldots, s\}  \  \text{
  s.t. } \left\langle
  m\left(\bm{\xi}- \bm{\theta_i}\right)+\bm{\zeta}\right\rangle_{\Z^d}
  < \vre.  
\end{equation} 
Then any line segment with length $M$ gets
$\vre$-close to a point in
$\mathfrak{F}\left(\bm{\Theta}_{s,d}\right)$.
\end{prop}

In turn, \eqref{distfibregeneralised} is implied by the statement that
for every $ \bm{\xi}\in \R^d$ there is an index $1 \le i\le s$ for
which the finite sequence 
$\left(m \cdot \pi\left(\bm{\xi}-
      \bm{\theta_i}\right)\right)_{0\le m \le M}$ 
is $\vre$-dense in $\T^d$ 
%
(with respect to 
the metric on $\T^d$ 
induced by $\langle \, \cdot \,\rangle_{\Z^d}$). 
We will now investigate conditions under which 
the multiples of a vector are \emph{not}
$\vre$-dense in the torus. Given
parameters $\vre>0$ and $M\ge 1$, define  
\begin{equation}\label{defC_eps_M}
C_d(\vre, M) = \left\{ 
\bm{\xi}\in\T^d : \text{ the sequence } \left(m
    \bm{\xi}\right)_{0\le m \le M} \text{ is not }
\vre\text{-dense in } \T^d 
\right\}. 
\end{equation}
Let also 
\begin{equation*}\label{defS_eps_M_c}
S_d(\vre, M)  \df \left\{\bm{\xi}\in\T^d\; :\; \exists
  \bm{u}\in\Z^d\backslash\left\{\bm{0}\right\}, \;\;
  \left\|\bm{u}\right\|\le c_d\vre^{-1}, \;\; \left\langle \bm{u\cdot
      \xi}\right\rangle \leq c'_d\cdot\frac{\vre^{d-1}}{M^{1/d}} \right\},
\end{equation*}
where 
\begin{equation}\label{defcd}
c_d\df d \qquad \textrm{and} \qquad c'_d \df d^{3/2}.
\end{equation}

\begin{prop}\label{propreduc}
With the above notation, assume that 
\begin{equation}\label{contrainteM}
M\ge 2^d  \vre^{-d}.
\end{equation} 
Then, $$C_d(\vre, M)\,\subset\, S_d(\vre, M).$$
\end{prop}

Proposition~\ref{propreduc} will be proved in the next section. We now
use it to derive Theorem~\ref{thmforestperes}. 

\begin{proof}[Deduction of Theorem~\ref{thmforestperes} from
  Proposition~\ref{propreduc}] Given $\vre>0$, $M\ge 1$ and
  $\bm{\Theta}_{s,d}$ as in~\eqref{defthetas}, set 
\[ \Sigma_d\left(\vre, M, \bm{\Theta}_{s,d}\right)\; \df\;
\bigcap_{i=1}^{s}\left(S_d(\vre, M)+\bm{\theta}_{i}\right), \]
where addition is taken on $\T^d$ and we identify $\bm{\theta}_i$ with
its projection modulo $\Z^d$.

Assume that $\bm{\Theta}_{s,d}\in
UDT_s^d(\Phi)$. 
Definition~\ref{defudt} is then readily seen to imply
that the set $\Sigma_d\left(\vre, M, \bm{\Theta}_{s,d}\right)$ is
empty whenever 
\begin{equation}\label{eq: whenever}
M>\left(c'_d\cdot \vre^{d-1}\cdot
  \Phi\left(c_d\vre^{-1}\right)^{-1}\right)^d,
\end{equation}
in which case, for every $\bm{\xi}$  there
exists an index $i\in\{ 1, \ldots , s\}$ such that $\bm{\xi}
-\bm{\theta}_{i}\not\in S_d(\vre, M)$. Using \equ{grothphi} we see that
\equ{eq: whenever}
implies that
$
M > c \vre^{-d}
$
for some constant $c$ depending only on $d$. Thus, replacing $M$ if
necessary by its constant multiple, we have that \eqref{contrainteM}
is also satisfied, and hence 
by Proposition~\ref{propreduc} we have that 
$\left(m \left(\bm{\xi}-
      \bm{\theta_i}\right)\right)_{0\le m \le M}$ 
is $\vre$-dense in $\T^d$. This implies Theorem~\ref{thmforestperes}
via Proposition \ref{prop: have shown}. 
\end{proof}

\section{Effective equidistribution in tori}\label{sec: toral equidistribution}
%
The goal of this section is to prove
Proposition~\ref{propreduc}. In this section, unless stated otherwise,
we continue with the 
notation $\|\bm{x}\| = \|\bm{x}\|_\infty$, and use the metric on
$\T^d$ induced by the sup-norm. 
%
The following lemma provides a necessary condition for
$\bm{\xi}\in\T^d$ to belong to the set $C_d(\vre, M)$ defined
in~\eqref{defC_eps_M}. This condition reduces the proof of
Proposition~\ref{propreduc} to the study of the 
multiples of a \emph{rational} vector. 

\begin{lem}\label{leminclusion}
Assume that~\eqref{contrainteM} holds.
Then
\begin{equation}\label{eq:inclusion}
C_d\left(\vre, M\right)\; \subset \; \bigcup_{\bm{p}/q\in S}
B\left(\frac{\bm{p}}{q}, \frac{1}{qM^{1/d}} \right),
\end{equation}
where
$S$ is the set of all rational vectors $\bm{p}/q\in\T^d$ such that  
\begin{equation}\label{reducratio}
1 \le q \le M \qquad \textrm{and} \qquad \frac{\bm{p}}{q}\in
C_d\left(\frac{\vre}{2}, q\right).  
\end{equation}
\end{lem}

\begin{proof}
We prove that the complement of the right hand side of \eqref{eq:inclusion} is contained in the complement of the left hand side.   

Let $\bm{\xi}\in\bigcap_{\bm{p}/q\in S}\limits\left[\T^d\smallsetminus B\left(\frac{\bm{p}}{q}, \frac{1}{qM^{1/d}} \right)\right]$. By Dirichlet's theorem, there exist a vector
$\bm{p}\in\Z^d$ and an integer 
$1\le q \le M$ such that  
$$\left\| \bm{\xi}-\frac{\bm{p}}{q}\right\|<
\frac{1}{qM^{1/d}}\cdotp$$ 
This implies that $\bm{p}/q\notin S$, namely that $\left(m\cdot \pi(\bm{p}/q )
\right)_{0\le m \le q-1}$ is $\vre/2$-dense in $\T^d$. Assuming \eqref{contrainteM}, we show that $\left(m\cdot\pi(\bm{\xi})  \right)_{0\le m \le M}$ is
$\vre$-dense in $\T^d$. 

Let $\bm{\lambda}\in\T^d$. 
By the $\vre/2$-density of 
$\left(m \cdot\pi(\bm{p}/q )\right)_{0\le m \le q-1}$
there exists an integer $0\le m \le q-1$ such that  
$$\left\langle m \frac{\bm{p}}{q} - 
\bm{\lambda} 
\right\rangle_{\Z^d} <\frac{\vre}{2}\cdotp$$ 
Then,
\begin{align*}
\left\langle m\bm{\xi}-\bm{\lambda} \right\rangle_{\Z^d} \; & = \;
                                                              \left\langle
                                                              m\left(\bm{\xi}-\frac{\bm{p}}{q}
                                                              \right)
                                                              +\left(
                                                              m\frac{\bm{p}}{q}
-\bm{\lambda}\right)\right\rangle_{\Z^d}
  \\ 
& \le \; m \cdot \left\| \bm{\xi}-\frac{\bm{p}}{q}\right\|+
  \left\langle m \frac{\bm{p}}{q} - 
\bm{\lambda} 
\right\rangle_{\Z^d}\\ 
& < \; \frac{1}{M^{1/d}}+\frac{\vre}{2}
\underset{\eqref{contrainteM}}{\le}\; \vre,
\end{align*}
whence the lemma.
\end{proof}

In view of Lemma~\ref{leminclusion}, we wish to provide a
necessary condition for the relation 
$\bm{p}/q\in
C_d\left(\frac{\vre}{2}, q\right)$ appearing
in~\eqref{reducratio} to
hold. For this we will recast the statement in terms of lattices. 
Let $\Lambda\left(\bm{p},
  q\right)$  be the lattice spanned by the rational vector
$\bm{p}/q\in\R^d$  and by the vectors  $\bm{e}_1, \dots, \bm{e}_d$ of
the standard basis of $\R^d$; that is,  
\begin{equation*}\label{defreseauratio}
\Lambda\left(\bm{p}, q\right) \;\df \;
\textrm{span}_{\Z}\left\{\frac{\bm{p}}{q},  \bm{e}_1, \dots,
  \bm{e}_d\right\} \; \subset \; \R^d . 
\end{equation*}
Also let 
\[ \Lambda^*\left(\bm{p}, q\right)\; \df \; \left\{\bm{u}\in\Z^d\;\; :
  \; \; \bm{p\cdot u}\equiv 0\pmod{q} \right\}. \]
It is easily seen that $\Lambda\left(\bm{p}, q\right)$
is the dual of $\Lambda^*\left(\bm{p}, q\right)$, and of index $q$ in
$\Z^d$. From this and \equ{eq: vol covol} it is easy to deduce the following:

\begin{lem}\label{Pproponewlattice}
The lattice $\Lambda\left(\bm{p}, q\right)$
has
covolume $1/q$ whenever $\gcd(\bm{p},q)=1$, and 
\begin{equation}\label{interpretationreseau}
\pi\left(\Lambda\left(\bm{p}, q\right) \right)  = 
\left(k\cdot \pi\left(\frac{\bm{p}}{q}\right)\right)_{0\le k\le q-1}. 
\end{equation}
\end{lem}

\ignore{
\begin{proof}
The first two claims are easily seen to hold. As for the last one, it
follows from~\cite[Corollary p.25]{Cassels} (with $n=d$, $b_n^*=q$ and
$\Lambda=\Z^n$) that the lattice $\Lambda_d^*\left(\bm{p}, q\right)$
has covolume $q$ under the assumption that $\gcd(\bm{p},q)=1$. This is
enough to conclude, as the covolume of a full rank lattice is the
inverse of the covolume of its dual. 
\end{proof}
}

Recall from \S \ref{subsec: lattices} that $\lambda_1(\Lambda)$ and
$\mu(\Lambda)$ denote respectively the first minimum and the 
covering radius of a lattice $\Lambda$. 
We now show:
\begin{lem}\label{lemfirstmindesity}
Assume that the Euclidean length
of the shortest nonzero vector in $\Lambda^*\left(\bm{p},q \right)$
satisfies 
\[ \lambda_1\left(\Lambda^*\left(\bm{p},q \right) \right)
>\; d \cdot \vre^{-1}. \]
Then the sequence $\left(k\cdot \pi\left(\bm{p}/q\right) \right)_{0\le k\le
  q-1}$ is $\vre/2$-dense in $\T^d$. 
\end{lem}


\begin{proof}
A well-known result of Banaszczyk \cite[Thm. 2.2]{bana},
asserts that 
for any lattice $\Lambda \subset \R^d$, $\mu(\Lambda) \cdot
\lambda_1(\Lambda^*) \leq d/2$ (we note that weaker results had been
known for some time, and these could also be used in our context,
at the sole expense of requiring a change in the constants appearing in \eqref{defcd}). 
Since $\Lambda(\bm{p},q)$
contains $\Z^d$, the sequence \eqref{interpretationreseau} will be
$\vre/2$-dense in $\T^d$ (with respect to the sup-norm) provided the
sup-norm covering radius of $\Lambda(\bm{p},q)$ is at most
$\vre/2$. Thus the Lemma follows immediately from Banaszczyk's bound
and the bound $\|\bm{x}\| \leq \|\bm{x}\|_2$. 
%
\end{proof}

We will need a further transference result
(see~\cite[Theorem~II, Chap.~V]{Casselsbis} for a proof): 
\begin{lem}[Mahler's Transference Theorem]\label{Mahler}
Let $\bm{\xi} \in \R^d$ and 
assume that
there is a nonzero integer $q$ such that  
$$\left\langle q\bm{\xi}\right\rangle_{\Z^d}\le C \quad \textrm{and}
\quad |q| \le U,$$ 
for real parameters $C$ and $U$ satisfying $0<C<1\le U$.
Then there is $\bm{v}\in\Z^d \setminus \{\bm{0}\}$ such that 
$$\left\langle \bm{\xi} \cdot \bm{v}\right\rangle \le D \quad
\textrm{and} \quad \left\|\bm{v}\right\| \le V,$$ 
where
$$D=d U^{-(d-1)/d}C, \quad V=dU^{1/d}.$$ 
\end{lem}

\begin{proof}[Proof of Proposition~\ref{propreduc}]
  Let  $\bm{\xi}\in C_d\left(\vre, M \right)$, where 
$M$ satisfies~\eqref{contrainteM}. From 
  Lemma~\ref{leminclusion}, there exist $\bm{p}\in\Z^d$ and $q\ge 1$
  such that 
  \begin{equation}\label{eq: 6.1 substitute}
  \left\| q\bm{\xi}-\bm{p}\right\|\le M^{-1/d}
  \end{equation}
  and such that~\eqref{reducratio} holds. 

Assume first that $q< \vre^{-d}$.
Lemma~\ref{Mahler}, applied with the parameters 
$$
 C=\frac{1}{M^{1/d}} \ \text{ and }  \ U=\vre^{-d},
$$ 
yields the existence of 
$\bm{v}\in\Z^d$ such that 
\[
 \left\langle \bm{\xi\cdot v}\right\rangle\; \le \; d\cdot \frac{\vre^{d-1}}{M^{1/d}}\qquad \textrm{ and } \qquad 
1\; \le \; \left\|\bm{v}\right\|\; \le \; d\cdot\vre^{-1} .
\]
In particular, $\bm{\xi}\in S_d(\vre,M)$. Assume now that 
\begin{equation}\label{secondcondiition}
q  \ge
\vre^{-d}. 
\end{equation}
Since $\bm{p}/q\in C_d(\frac{\vre}{2},q)$, 
Lemma~\ref{lemfirstmindesity} implies the existence of 
$\bm{u}\in\Lambda^*(\bm{p},q) \subset \Z^d$ with  
\begin{equation}\label{normubound}
1\; \le \; \left\|\bm{u}\right\| \; \le \; \left\|
 \bm{u}\right\|_2\; \le \; d\cdot \vre^{-1}, 
\end{equation}
and hence  such that, by \eqref{eq:Dual_lattice},
$$\frac{\bm{p}}{q}\bm{\cdot u}=k$$
 for some integer $k$. Then by the Cauchy-Schwarz inequality, 
 \begin{align*}
 \left\langle \bm{\xi\cdot u} \right\rangle \; & \le \;
                                                 \left|\bm{\xi\cdot
                                                 u}-k \right| \; = \;
                                                 \left|\left(\frac{\bm{p}}{q}-\bm{\xi}
                                                 \right) \bm{\cdot u}
                                                 \right| \nonumber \\ 
 &\le \; \left\| \frac{\bm{p}}{q}-\bm{\xi}\right\|_2 \cdot \left\|
   \bm{u}\right\|_2 \nonumber \\ 
 & \underset{\eqref{eq: 6.1 substitute} , \eqref{normubound}}{\le} \;
                                         \frac{\sqrt{d}}{qM^{1/d}}\cdot
                                         d\cdot \vre^{-1} \nonumber
   \\ 
  & \underset{\eqref{secondcondiition} }{\le} \; 
    d^{3/2}\cdot\frac{\vre^{d-1}}{M^{1/d}}. \label{equadefinitive} 
 \end{align*}
%
%
By \eqref{defcd},  $\bm{\xi}\in S_d(\vre,M)$ and the proof of Proposition \ref{propreduc} is complete. 
\end{proof}

\subsection{An explicit uniformly discrete dense forest}\label{subsec:
  three lattice construction} 
In this section we prove Theorem \ref{thm: finitely many lattices}. We
will work with a variant of the set $\mathfrak{F}(\bm{\Theta}_{s,d})$, which
can be analyzed in a similar way. Let $\alpha, \beta, \gamma,
\delta$ be nonzero real numbers satisfying the
following conditions:

\begin{itemize}
\item[(i)] $
  \frac{\gamma}{\delta(\alpha + \gamma)} \in \Q$. 
\item[(ii)] $(\alpha+ \gamma)\, (\beta + \delta)=1$. 
\item[(iii)] 
For any $\eta>0$ there is $c>0$ such that for any integers $P, Q$, not
both zero, 
$$
\langle P \alpha + Q \gamma \rangle \geq c \max\{|P|,|Q|\}^{-(2+\eta)}
$$
and 
$$
\langle P \beta + Q \delta \rangle \geq c \max\{|P|,|Q|\}^{-(2+\eta)}. 
$$

\end{itemize} 
Now define 
$$
\Lambda_1 = \Z^2, \ \ \Lambda_2 = \left(\begin{matrix} \gamma &
    \alpha \\ 0 & 1 \end{matrix}
\right) \cdot \Z^2, \ \ \Lambda_3 = \left(\begin{matrix} 1 & 0 \\
    \beta & \delta\end{matrix}
\right) \cdot \Z^2.
$$
Theorem \ref{thm: finitely many lattices} follows from the following
statements:

\begin{prop}\label{prop: for three lattices1}
Assuming \rm{(i)}, \rm{(ii)}, there are $\bm{x}_2, \bm{x}_3 \in \R^2$ such that $\Lambda_1 \cup
(\bm{x}_2 + 
\Lambda_2) \cup (\bm{x}_3 + \Lambda_3)$ is uniformly discrete. 
\end{prop}

\begin{prop}\label{prop: for three lattices2}
Assuming \rm{(iii)}, for any $\eta>0$ there is $c>0$ such that for any $\vre>0$, any
$\bm{x} \in \R^2$, any $M > c/\vre^{5+\eta}$, and any slope 
$\sigma \in [-1,1]$, 
the set $\Lambda_1 
\cup (\bm{x} + \Lambda_2)$ comes within $\vre$ of any `nearly
vertical' line segment 
$$
\{\bm{y} + t \bm{u} : t \in [0, M]\}, \ \ \text{ where } \bm{u} =
(\sigma, 1)^T.
$$  
A similar statement holds replacing $\Lambda_2$ with $\Lambda_3$ and
$\bm{u}$ with $(1,\sigma)^T$ (that is, $\Lambda_1 \cup (\bm{x} + \Lambda_3)$
comes $\vre$-close to `nearly horizontal' segments). 
\end{prop}

\begin{prop}\label{prop: third}
There are examples of numbers $\alpha, \beta, \gamma, \delta$
satisfying hypotheses \rm{(i)}--\rm{(iii)}. 
For instance, one can define 
$$\alpha \df \sqrt{2}, \ \
\beta \df 3-\sqrt{2}+\sqrt{3}-\sqrt{6}, \ \ \gamma \df \sqrt{3}, \ \ \delta
\df -3 + \sqrt{6}.
$$ 
\end{prop}

\begin{proof}[Proof of Proposition \ref{prop: for three lattices1}]
In light of Proposition \ref{prop: uniformly discrete unions}, it is
enough to show that none of the three sets 
$$\Lambda_1 + \Lambda_2, \ \Lambda_1 +
\Lambda_3, \ \Lambda_2 + \Lambda_3$$ 
are dense in $\R^2$. This is clear for $\Lambda_1 + \Lambda_2$
(respectively, 
$\Lambda_1 + \Lambda_3$), since the second (resp. first) coordinate of
any vector in this set is an integer. For $\Lambda_2 + \Lambda_3$ we
note that by (ii), 
$$
\det \left(\begin{matrix} \alpha + \gamma & 1 \\ 1 & \delta +
    \beta \end{matrix}  \right) 
    = 0,
$$
and hence the two vectors $\left(\alpha + \gamma, 1\right)^T \in \Lambda_2, \
\left(1, \delta + \beta \right)^T \in \Lambda_3$ are
collinear. Let $\ell$ denote the line perpendicular to $\left(\alpha + \gamma, 1\right)^T$. Note that $\bm{u}_2 = (\gamma, 0)^T \in \Lambda_2$ and
$\bm{u}_3 = (0, \delta)^T \in \Lambda_3$. Then the following
calculation shows that the projections of $\bm{u}_2, \bm{u}_3$ onto $\ell$ are nonzero and
commensurable:
$$
\frac{\bm{u}_2 \cdot \left( -1, \alpha+\gamma\right)^T}{\bm{u}_3\cdot
  \left( -1, \alpha+\gamma\right)^T} = \frac{-\gamma}{\delta(\alpha +
  \gamma)} 
\in \Q. 
$$

This implies that the projection of $\Lambda_2 + \Lambda_3$ onto $\ell$ is not
dense and in particular $\overline{\Lambda_2 + \Lambda_3} \neq \R^2$. 
\end{proof}

\begin{proof}[Proof of Proposition \ref{prop: for three lattices2}]
We work with $\Lambda_1 \cup (\bm{x}+ \Lambda_3)$ and `nearly
horizontal' segments, the proof for nearly vertical segments being
similar. 
Let 
$$\mathcal{L} = \left\{ \ell(t) : t \in [0, M] \right \},  \
\text{ where } \ell(t) = (y_1 +t, y_2 + \sigma t)^T \in \R^2,$$ with $$
y_1, y_2, \sigma \in \R \
\text{ and } \ |\sigma| \leq 1. $$
Also let $\bm{x} = (x_1, x_2)^T$. 

As we saw in the proof of Proposition \ref{prop: have shown}, if
\begin{equation}\label{eq: one density}
( m \cdot \pi\left(\sigma \right))_{1 \leq m \leq M} \text{ is } \vre\text{-dense in }\T^1
\end{equation}
then
$\Lambda_1$ comes $\vre$-close to $\mathcal{L}$. 
By a similar argument, if 
\begin{equation}\label{eq: second density}
\left(
m\cdot \pi \left(\frac{\sigma-\beta}{\delta}\right)\right)_{1 \leq m \leq
M}\text{ is } \frac{\vre}{\delta}\text{-dense in }\T^1
\end{equation} 
then $\bm{x}+\Lambda_3$ comes $\vre$-close to $\mathcal{L}$. 
Indeed, setting 
$t_m \df m - \{y_1- x_1\}$ (where $\{ x\}$ denotes the fractional part of $x\in\R$), $j=j_m= m+\lfloor y_1-x_1\rfloor$ (where $\lfloor x\rfloor$ denotes the integer part of $x\in\R$), we have that
$$\ell(t_m) = \left( \begin{matrix} y_1 -\{y_1- x_1 \} + m \\ y_2 - \sigma \{y_1 -
x_1 \}+ \sigma m \end{matrix}\right)= \left(\begin{matrix} x_1 + j \\
y_2 - \sigma \{y_1 - 
x_1 \}+ \sigma m \end{matrix} \right)$$
is $\vre$-close to 
$$
\bm{x}+\Lambda_3 = \left\{\left( \begin{matrix} x_1 + j \\ x_2 + j\beta +
    k\delta\end{matrix} \right): j,k \in \Z \right\}
$$
when 
$$\left\langle \frac{1}{\delta} \left(y_2 - \sigma \{y_1 - x_1\} -x_2 -
    \beta \lfloor y_1-x_1 \rfloor\right)+m
    \left(\frac{\sigma-\beta}{\delta} \right) \right \rangle < \frac{\vre}{\delta}.$$

  So it remains to show that for $M > c/\vre^5$, for any $\sigma \in \R$,
  at least one of \equ{eq: one density}, \equ{eq: second density}
  holds. If not, then by Lemma \ref{leminclusion} there are $q_1,
  q_2 \in \Z$ with $1 \leq |q_1|\le (\epsilon/2)^{-1}$,  $1\le 
  |q_2| \leq (\vre/2\delta)^{-1}$ and $p_1, p_2 \in \Z$ such that 
$$| q_1\sigma - p_1|< \frac{1}{M} \ \ \text{ and } \ \left|q_2
  \left(\frac{\sigma-\beta}{\delta} \right) - 
  p_2 \right| < \frac{1}{M}$$ 
(note indeed that the set $C_d(\eta, q)$ appearing in Lemma  \ref{leminclusion} is easily described when $d=1$~: it is the set of rationals $p/q$ such that $1\le |q|\le \eta^{-1}$ whenever $\gcd(p,q)=1$).
Multiplying the first formula by $q_2$ and the second one by
$q_1 \delta$ and using the triangle inequality we obtain
\begin{equation}\label{eq: cannot hold}
\left| q_1q_2 \, \beta +q_1 p_2 \delta - p_1 q_2\right| < \frac{2 \delta}{\vre M}.
\end{equation}
Now set $P = q_1q_2, \, Q = q_1p_2$, and invoke assumption (iii),
with $\eta/2$ in place of $\eta$.
At the possible expense of replacing $M$ with its 
constant multiple, we see that \equ{eq: cannot hold} cannot happen
when $M > c/\vre^{5+\eta}$.  
\end{proof}

\begin{proof}[Proof of Proposition \ref{prop: third}]
It is easy to check that (i) 
and (ii) are satisfied by $\alpha, \beta, \gamma, \delta$. With these
choices, $\alpha $ is an irrational in $\Q(\sqrt{2})$, $\gamma$
is an irrational in $\Q(\sqrt{3})$, and $\beta, \delta$ are in
$\Q(\sqrt{2}, \sqrt{3})$ such that $1, \beta, \delta$ are linearly
independent over $\Q$. Now requirement (iii) follows from a theorem of
Schmidt, see \cite[Cor. 1E, p.152]{SchmidtDioph}. 
\end{proof}

\section{A metric theory of uniformly Diophantine $s$--tuples}\label{sec:
  diophantine  
  notions} 
Throughout this section, $\Phi$ is a non-increasing function tending
to zero at infinity, $s\ge d+1$, and $\bm{\Theta} = \bm{\Theta}_{s,d}$ is an
$s$-tuple of vectors in $\R^d$.   
  
\subsection{Uniformly Diophantine $s$-tuples and multilinear
  algebra} \label{udtmultilin} 
Our goal is to provide a sufficient condition for $\bm{\Theta}$
to belong to the set $UDT_s^d(\Phi)$. We will require some
preliminaries from multilinear algebra. We introduce the required
notions and facts, referring to~\cite[Chap.3]{bourbaki} for proofs and
more details. 

Equip $\R^{s}$ with its usual scalar product and let
$\left\{\bm{e}_i\right\}_{1\le i\le   s}$ be the standard basis. The
Grassmann algebra is the vector space 
$$\bigwedge\R^{s}\; \df \; \bigoplus^{s}_{r=0}
\bigwedge^{r}\R^{s}$$ 
equipped with the inner product for which the set of wedge products  
$\bm{e}_{i_1}\wedge \dots \wedge \bm{e}_{i_r},$
where $$1\le i_1 < i_2 <  \,\dots\, < i_r \leq s \qquad \mbox{ and }\qquad
0\le r\le s,$$ is an orthonormal basis.  A multivector 
$\bm{X}\in 
\bigwedge^{r}\R^{s}$ is said to be {\em decomposable} if
there exist $\bm{x}_1, \, \dots\, , \bm{x}_r$ in
$\R^{s}$ such that $\bm{X} = \bm{x}_1\wedge \dots\wedge \bm{x}_r$.  
The Cauchy-Binet formula shows that the scalar product $\bm{X\cdot Y}$
between two pairs of decomposable vectors $\bm{X} = \bm{x}_1\wedge
\dots\wedge \bm{x}_r$ and $\bm{Y} = \bm{y}_1\wedge \dots\wedge
\bm{y}_r$ 
is given
by $$\bm{X\cdot Y} = \det\left(\bm{x_i\cdot y_j}\right)_{1\le i,j\le
  r}.$$ 
From now on, the notation $\left\|\: .\:\right\|$ will be reserved for the norm derived from this inner product  (note that its restriction to $\bigwedge^1 \R^s \simeq \R^s$ is the usual Euclidean norm $\left\|\: .\:\right\|_2$ in $\R^s$).

Let $\mathbb{P}(\bigwedge \R^s)$ be the space of lines in $\bigwedge
\R^s$, and for any subspace $V$ of $\R^{s}$, given a basis $\bm{v}_1,
\dots, \bm{v}_r$ of $V$, 
define $\bm{X}_V \in \mathbb{P}(\bigwedge \R^s)$ as the line spanned
by $\bm{v}_1 \wedge \dots \wedge \bm{v}_r\in
\bigwedge^{r}\R^{s}.$ It is easily seen that this is well-defined
(independent of the choice of the basis), and it is known that the map
$V \mapsto \bm{X}_V$ (which is called the \emph{Pl\"{u}cker
  embedding})  is a bijection between the set of $r$-dimensional
linear subspaces in $\R^{s}$ and 
the set of lines spanned by nonzero decomposable multivectors in $
\bigwedge^{r}\R^{s}$. For any nonzero $\bm{u} \in  \R^s$, the length
of the projection of $\bm{u}$ on the space orthogonal to $V$ is given by 
$$
\|\bm{X}_V \wedge \bm{u}\| \df \frac{\| \hat{\bm{X}}_V  \wedge \bm{u}
  \|}{\| \hat{\bm{X}}_V \|}, \ \ 
\text{ where } \hat{\bm{X}}_V = \bm{v}_1 \wedge \cdots
  \wedge \bm{v}_r, 
$$
and this is again independent of choices. The quantity
$\frac{\|\bm{X}_{V} \wedge \bm{u}\|}{\|\bm{u}\|}$ is 
sometimes called the {\em projective distance} between $V$ and the line spanned by 
$\bm{u}$. 
See \cite[\S3]{bugeaudlaurent} and~\cite[\S2]{laurent} for more
details. 

Given an integer $T\ge 1$, define $\mathcal{V}_{s,d}(T)$ to be the set of
$s\times d$ integer matrices
\begin{equation}\label{defvsdT}
\left\{(\bm{u}_1, \dots,
  \bm{u}_s)^T\in\Z^{s\times d} : \forall i\in\{1, \ldots, 
  s\}, \ \bm{u}_i  \in \Z^d \text{ and } 1\le \left\|\bm{u}_i\right\|_{\infty}\le
  T\right\}.
\end{equation}
Furthermore, given a matrix $\bm{U} = (\bm{u}_1, \dots,
\bm{u}_s)^T\in\mathcal{V}_{s,d}(T)$, define 
\begin{equation}\label{deftuthetasd}
\bm{t}_{\bm{U}}(\bm{\Theta}) \df
\left(\bm{u}_1\bm{\cdot\theta}_{1}, \, \dots\, ,
  \bm{u}_s\bm{\cdot\theta}_{s}\right)^T \in \R^s
\end{equation}
and set for simplicity $\bm{X}_{\bm{U}} =
\bm{X}_{\textrm{colspan}(\bm{U})}$, where $\textrm{colspan}(\bm{U})$
is the subspace of $\R^s$ spanned by the colums of the matrix
$\bm{U}$. 

The main result in this section is then the following:

\begin{prop}\label{propmultilinalg}
Assume that $\bm{\Theta} \notin UDT_s^d(\Phi)$. Then there
exist $T\ge 1, \ \bm{p} \in \Z^s$ and $\bm{U}\in\mathcal{V}_{s,d}(T)$
such that  
\begin{equation}\label{eq: 5.36}
|p_i| \leq 4 \sqrt{d} \cdot \|\bm{u}_i\|_2 \cdot \max
  \{1, \|\bm{\theta}_i \|_\infty\}
\end{equation}
for all $i \in \{ 1, \ldots, s\}$ 
and
\begin{equation}\label{defyputheta}
y_{\bm{p} }\left(\bm{U},  \bm{\Theta}\right)\; = \; \bm{p}+
\bm{t}_{\bm{U}}(\bm{\Theta}) 
\end{equation}
satisfies 
\begin{equation}\label{fundinegmultiudt}
\left\|\bm{X}_{\bm{U}} \wedge 
y_{\bm{p} }\left(\bm{U},
      \bm{\Theta}\right)\right\|
\; < \;  \sqrt{s} \cdot \Phi(T).
\end{equation} 
\end{prop}

In particular, $\bm{\Theta}\in UDT_s^d(\Phi)$ as soon as 
\begin{equation}\label{conditionudt}
\inf_{T\ge 1} \; \min_{\bm{U}\in\mathcal{V}_{s,d}(T)} \;
\inf_{\bm{p}\in\Z^{s}} \; 
\left( \sqrt{s} \cdot
    \Phi(T)\right)^{-1}\cdot 
\|\bm{X}_{\bm{U}} 
\wedge
      y_{\bm{p} }\left(\bm{U},
        \bm{\Theta}\right)\|
\; \ge \; 1. 
\end{equation}
%

This condition should be compared with those appearing in the theory
of approximation of vectors by rational subspaces. Let
$\bm{y}\in\R^{s}$ be a nonzero vector. In the standard theory (see
\cite{bugeaudlaurent, laurent} and the references therein), one
is interested in showing the existence of rational $s\times d$ matrices $\bm{U}$ of a given rank $1\le r\le d$ for which the
inequality $\|\bm{X}_{\bm{U}} \wedge \bm{y}\|\le \Phi(T)$ holds under the assumption that
the so--called  `Weil height' of the subspace $\textrm{colspan}(\bm{U})$ is bounded by $T$ (this height is at most $(T')^r$ if the columns of $\bm{U}$ have Euclidean norms at most $T'$). In the problem we are considering, the vector $\bm{y}$ is
not fixed but rather varies along with the approximant $\bm{U}$, via
formula \eqref{defyputheta}. 

Proposition~\ref{propmultilinalg} justifies a claim made
after Theorem~\ref{thmforestperes}; namely that a pair $\left(\alpha,
  \beta\right)^T\in\R^2$ such that  $\sigma \df \beta -
\alpha$ is a badly approximable number belongs to 
$UDT_2^1(3)$. Indeed, 
$$\mathcal{V}_{2,1}(T)\;=\;\left\{(q, v)^T\in\Z^2\; :\; 1\le
  \left|q\right|, \left|v\right| \le T\right\}.$$ 
From condition~\eqref{conditionudt}, the claim is easily seen to
be implied by  the existence of a constant $c=c\left(\sigma \right)>0$
such that for 
$T\ge 1$,
$$\min_{1\le |q|, |v|\le T}\frac{\langle
  qv\sigma\rangle}{\sqrt{q^2+v^2}}\ge \frac{c}{T^3}\cdot$$ 
This follows from the assumption that $\sigma$ is a badly
approximable number; that is, from the relation $\inf_{m\in
  \Z\backslash \{0\}}|m|\cdot\langle m\sigma\rangle\; >\; 0.$

\begin{proof}[Proof of Proposition~\ref{propmultilinalg}]
The condition  $\bm{\Theta} \not\in UDT_s^d(\Phi)$  means that
there exist $T\ge 1$ and $\bm{\xi}\in\R^d$  such that for each index
$1\le  i \le s$, one can find an integer $p_i$ and an integer vector
$\bm{u}_i$ satisfying the relations 
\begin{equation}\label{systeq}
1\le \left\|\bm{u}_i \right\|_{\infty}\le T \ \ \mbox{ and } \ 
\bm{u}_i\bm{\cdot \xi} = p_i+\bm{u}_i\bm{\cdot\theta}_{i}+\delta_i, \
\text{ with } \ |\delta_i|<\Phi(T).
\end{equation}
The $p_i$ here satisfy the bound \equ{eq: 5.36}. Indeed, 
we may assume (translating $\bm{\xi}$ by an integer vector if
necessary) that $\left\|
  \bm{\xi}-\bm{\theta}_i\right\|_{\infty}\le 1 $. Thus for any 
$1\le i\le s$, by the Cauchy--Schwarz inequality, 
\begin{align*}
\left| p_i \right|\; & \le \; \left\|\bm{u}_i\right\|_2 \cdot \left(
                       \left\|\bm{\xi}\right\|_2
                       +\left\|\bm{\theta}_i\right\|_2\right)+\left|
                       \delta_i\right| \nonumber \\ 
& \le \; \left\|\bm{u}_i\right\|_2 \cdot \sqrt{d} \cdot \left( 1+2
  \left\|\bm{\theta}_i\right\|_{\infty}\right)+1, 
\end{align*}
where the trivial bound $\Phi(T)\le 1$ guaranteed
by~\eqref{lowboundudt} is used to obtain the last inequality. 
Now define 
\begin{itemize}
\item $\bm{U}$ as the $s\times d$ matrix $\bm{U}\df (\bm{u}_1, \dots,
  \bm{u}_s)^T;$ 
\item $\bm{p}$ as the $s$-dimensional integer vector  $\bm{p} \df
  (p_1, \dots, p_s)^T;$ 
\item $\bm{\delta}$ as the $s$-dimensional vector $\bm{\delta} \df
  (\delta_1, \dots, \delta_s)^T;$ 
\item $\bm{t}_{\bm{U}}(\bm{\Theta})$ as the $s$-dimensional
  vector~\eqref{deftuthetasd}. 
\end{itemize}
The system of equations~\eqref{systeq} can then be rewritten as
\begin{align}\label{matrixsysteq}
\bm{U}\bm{\xi} = \bm{p}+ \bm{t}_{\bm{U}}(\bm{\Theta}) +\bm{\delta}
\end{align}
with 
\begin{align}\label{condmatrixsysteq}
\bm{U}\in\mathcal{V}_{s,d}(T) \qquad \textrm{ and } \qquad \left\|\bm{\delta}\right\|_{\infty}< \Phi(T).
\end{align}
Consider $\bm{\xi}$ as the unknown in the linear system of $s$
equations in $d$ variables~\eqref{matrixsysteq}. Assume furthermore
that 
$\bm{U}$ has rank  
$1\le r \le d,$
and let $\bm{v}_1, \, \dots\, , \bm{v}_r\in\Z^{s}$ denote $r$
linearly independent columns of the matrix $\bm{U}$. From the theory
of Gaussian elimination, the system~\eqref{matrixsysteq} admits a
solution if and only if $\bm{p}+
\bm{t}_{\bm{U}}(\bm{\Theta}) +\bm{\delta} \in \mathrm{span}
\left\{\bm{v}_1, \, \dots\, , \bm{v}_r \right\}$; that is, if
and only if
\begin{equation*}\label{wedgeconditiom}
\left(\bigwedge_{i=1}^{r} \bm{v}_{i} \right)\wedge \left(\bm{p}+
  \bm{t}_{\bm{U}}(\bm{\Theta}) +\bm{\delta}\right)\; = \;
\bm{0}. 
\end{equation*}
This equation can be rewritten as 
\begin{equation*}\label{wedgeconditiombis}
\hat{\bm{X}}_{\bm{U}}\wedge \left(\bm{p}+
  \bm{t}_{\bm{U}}(\bm{\Theta})\right)\; = \;
-\hat{\bm{X}}_{\bm{U}}\wedge \bm{\delta}, 
\end{equation*}
where $\hat{\bm{X}}_{\bm{U}} = \bm{v}_1 \wedge \, \dots\, \wedge
\bm{v}_r$. 
Hadamard's
inequality (see~~\cite[eq.~(13) p.49]{whitney}) then implies that
\begin{align*}\label{wedgeconditiomter}
\left\|\hat{\bm{X}}_{\bm{U}}\wedge \left(\bm{p}+
  \bm{t}_{\bm{U}}(\bm{\Theta})\right)\right\|\; &\le \;
                                                        \left\|\hat{\bm{X}}_{\bm{U}}\right\|\cdot
                                                        \left\|\bm{\delta}\right\|_{2}\\  
& \underset{\eqref{condmatrixsysteq}}{<}\; \sqrt{s} \cdot
  \Phi(T)\cdot \left\|\hat{\bm{X}}_{\bm{U}}\right\| , 
\end{align*}
whence the Proposition.
\end{proof}

\subsection{Towards a metric theory of uniformly Diophantine
  $s$-tuples}\label{proofthm5.3} 
The goal of this section is to establish
Theorem~\ref{thmforestperesbis}. This will be done with the help of
several lemmas. 

\begin{lem}\label{lemprojmultivec}
Let $r \in \{1, \ldots, s \}$, let $\bm{X}\in  \bigwedge^r \R^{s}$ be a nonzero
decomposable multivector, and let
$\bm{x}\in\R^{s}$. Then 
$$\left\|\bm{X}\wedge\bm{x}\right\| =
\left\|\bm{X}\right\| \cdot
\left\|P_{\bm{X}}^{\perp}\left(\bm{x}\right)\right\|, $$ where
$P_{\bm{X}}^{\perp}$ denotes the orthogonal projection onto the
orthocomplement of the subspace represented by 
$\bm{X}$. 
\end{lem}

\begin{proof}
This is well--known. See~\cite[Chap.1. \S15]{whitney} for details.
\end{proof}


The following is an easy consequence
of the compactness of the Grassmann variety of $k$-dimensional
subspaces in $\R^s$. Note that an explicit value of the constant $c_s$ below can be worked out from~\cite[Theorem~1]{miaoadi} (one can for instance take $c_s=2^{-s-1}$).

\begin{lem}\label{lemindiceproj}
There is a constant $c_s>0$ such that for any $k \in \{1, \ldots, s\}$
and any 
$k$-dimensional subspace $H \subset \R^{s}$, the following
holds. Denote by $P_H$ the orthogonal projection onto $H$ and by 
$\left(\bm{e}_1, \, \dots\, , \bm{e}_s\right)$ the standard basis of
$\R^{s}$. Then there exist indices $1\le i_1<i_1<\dots < i_k\le
s$ such that for any $\bm{x} \in \mathrm{span} \{\bm{e}_{i_j}\}_{1
  \leq j \leq k}$, 
$$\left\| P_H\left( \bm{x}\right) \right\|_2\; \ge c_s \|\bm{x}\|_2.$$
\end{lem}


We will also need a consequence of the Brunn-Minkowski inequality (see
\cite[\S 10.1]{schneider} for a more detailed discussion). 

\begin{lem}\label{covariogramme}
Let $C$ and $K$ be centrally symmetric convex bodies in $\R^k$.
Then for any
$\bm{x}\in\R^k$, 
\begin{equation}\label{eq: lemma 7.4}
\Vol \left(C\cap \left(K+\bm{x} \right) \right)\;\le
\; \Vol \left(C\cap K \right).
\end{equation}
\end{lem}

\begin{proof}
Let 
$$
\mathcal{C} = \{\bm{y} \in \R^k: C \cap (K+\bm{y}) \neq \varnothing\},
$$
and let 
$$f_{C,K}: \mathcal{C} \to \R, \ \ \ \ f_{C,K}(\bm{x}) \df \Vol\left(C\cap \left(K+\bm{x}
   \right)\right)^{1/k}.$$ 
   
   Fix $t\in (0,1)$ and $\bm{x}, \bm{y}\in \mathcal{C}$. Since 
\begin{align*}
C\cap\left(K+(1-t)\bm{x}+t\bm{y}) \right)\; & = \; C\cap\left((1-t)\left( K+\bm{x}\right)+t\left(K+\bm{y} \right)\right)\\
& \supset \; (1-t)\left( C\cap \left(K+\bm{x} \right)\right) + t\left(C\cap \left( K+\bm{y}\right)\right),
\end{align*}
   the Brunn-Minkowski inequality implies that $f_{C,K}$ is concave on $\mathcal{C}$.
If $\bm{x} \notin \mathcal{C}$ then \equ{eq: lemma 7.4} is immediate,
so let $\bm{x} \in \mathcal{C}$.  
Since $C$ and $K$ are centrally symmetric, we have $f_{C, K}(-\bm{x})
= f_{-C,
  -K}(-\bm{x})= f_{C, K}(\bm{x})$, and thus the concave function 
$$t\in [-1, 1]\mapsto f_{C, K}\left(t\bm{x}\right)$$ is
even. It therefore reaches its maximum when $t=0$.
\end{proof}

With the notation of Proposition~\ref{propmultilinalg}, given positive integers
$N, T$ and $\bm{U}\in \mathcal{V}_{s,d}(T)$, let
$E^{(N)}_{s,d}\left(\bm{U},T\right) $
be the set of $d\times s$-matrices $\bm{\Theta}$ satisfying 
\begin{equation}\label{normtheta}
\| \bm{\Theta} \|_\infty < N 
\end{equation}
and such that for some $ \bm{p}\in\Z^{s}$, \eqref{fundinegmultiudt}
holds, and 
\begin{equation}\label{conditiontheta} 
 \left|p_i\right|\le 4\sqrt{d}N\left\|\bm{u}_i \right\|_2 \ \ \text{
   for all }  i \in \{1,
 \ldots, s\}. 
\end{equation}
Note that this is just a reformulation of 
inequality~\equ{eq: 5.36} taking into account 
assumption~\eqref{normtheta}.  Then we have:

\begin{lem}\label{lememasure}
With the above notation, 
$$\Vol\left(E_{s,d}^{(N)}\left(\bm{U},T\right) \right ) = O\left(T^{r}\cdot
  \Phi(T)^{s-r}\right),$$ 
where $r = \mathrm{rank}(\bm{U})$ and the
implicit constant depends on 
$s, \, d$ and $N$.
\end{lem}

\begin{proof} Write the $d\times s$ matrix $\bm{\Theta} = \bm{\Theta}_{s,d}\in
E_{s,d}^{(N)}\left(\bm{U},T\right) $ as in~\eqref{defthetas}, and let 
$\bm{u}_1, \dots, \bm{u}_s\in\Z^{d}\backslash\{\bm{0}\}$
denote the transposes of the (nonzero) rows of $\bm{U}$. Fix an
integer vector $\bm{p} \df (p_1, 
\dots, p_s)^T\in\Z^{s} $ for which 
\eqref{conditiontheta} holds. 

Each $\bm{\theta}_i, \, i\in\{1, \ldots , s\}$ can be written uniquely
as 
\begin{equation}\label{decomptheta_i}
\bm{\theta}_i =
\lambda_i\cdot\frac{\bm{u}_i}{\left\|\bm{u}_i\right\|_2}+\bm{w}_i
\;\in\; B_{\infty}\left(\bm{0}, N \right), \qquad \textrm{ with }
\ \bm{w}_i\bm{\cdot}\bm{u}_i=0, \ \lambda_i \in \R.
\end{equation}
By the orthogonality in \eqref{decomptheta_i},  
upon identifying 
$\bm{u}_i^{\perp}$ with $\R^{d-1}$, the volume element on $\R^d$ can
be decomposed in the  coordinates~\eqref{decomptheta_i} as 
\begin{equation}\label{dtheta}
d \bm{\theta}_i = d\lambda_i \cdot d\bm{w}_i,
\end{equation}
and moreover 
\begin{equation}\label{sizelabmda8i}
\left|\lambda_i \right|< \sqrt{d}N \qquad \textrm{ and } \qquad
\left\|\bm{w}_i\right\|_{\infty} < \sqrt{d}N. 
\end{equation}
From the condition $\bm{\Theta} \in
E_{s,d}^{(N)}\left(\bm{U},T\right) $ we will derive a restriction on the
coefficients $\bm{\lambda} = (\lambda_i)_{i=1, \ldots, s}$; for the
vectors $\bm{w}_i$ we will not have any further restriction beyond the
bound on the right-hand side \eqref{sizelabmda8i}, i.e., they are 
bounded by constants depending only on $d$ and $N$. 

Let $H(\bm{U})$ be the orthocomplement of the subspace of $\R^{s}$
spanned by the columns of $\bm{U}$, so that $H(\bm{U})$  has
dimension $s-r$,
and let $I_{\bm{U}}\subset \{ 1, \ldots, s\}$ be the set of
$s-r$ indices obtained when applying Lemma~\ref{lemindiceproj}  to
$H(\bm{U})$. 
 Denoting by $P$ the orthogonal projection onto
$H(\bm{U})$, 
 Lemma~\ref{lemprojmultivec} and inequality~\eqref{fundinegmultiudt}
 imply that  
\begin{align}
  \left\|P\left(y_{\bm{p}}\left(\bm{U},\bm{\Theta}\right)\right)\right\|
  & = \; \left\|\bm{X}_{\bm{U}}\wedge y_{\bm{p} }\left(\bm{U},
  \bm{\Theta}\right)\right\|
  \nonumber \\ 
& < \; \rho \df \sqrt{s}\cdot \Phi(T). \label{defrho}
\end{align}
In terms of the standard basis $\bm{e}_1, \, \dots\, ,
\bm{e}_s$ of $\R^{s}$ and using 
\eqref{deftuthetasd}, \eqref{defyputheta} and 
\eqref{decomptheta_i}, this yields 
\begin{equation*}
\left\|P\left(y_{\bm{p}}\left(\bm{U},\bm{\Theta}\right)\right)\right\|  \; =\; \left\|  P\left(\sum_{i=1}^s\left(\lambda_i\left\|\bm{u}_i
      \right\|_2 +p_i\right)\, \bm{e}_i \right)\right\|_2 
\; < \; \rho,
\end{equation*}
which we can rewrite as
\begin{equation}\label{ineqnormratiobis}
 \left\|  \sum_{i\in I_{\bm{U}}} \lambda_i\left\|\bm{u}_i \right\|_2
   \cdot  P\left(\bm{e}_i \right)+\bm{x}\right\|_2\; < \; \rho, 
\end{equation}
where 
$$\bm{x}\; = \; \sum_{i\in I_{\bm{U}}}p_i\cdot   P\left( \bm{e}_i
\right) + \sum_{i\not\in I_{\bm{U}}}\left(\lambda_i\left\|\bm{u}_i
  \right\|_2 +p_i\right)\cdot  P\left( \bm{e}_i \right) \; \in
H(\bm{U}).$$ 
Define the centrally symmetric polytope 
$$C_{\bm{\lambda}}\; \df\; \left\{
  \sum_{i\in I_{\bm{U}}} \lambda_i\left\|\bm{u}_i \right\|_2 \cdot
  P\left(\bm{e}_i \right) \; : \; \left| \lambda_i\right| < \sqrt{d}N
  \; \textrm{for all }\; i\in I_{\bm{U}}\right\}.$$
Then 
\eqref{ineqnormratiobis} shows that for $\bm{\Theta} \in
E_{s,d}^{(N)}\left(\bm{U},T\right) $, the coefficients $\bm{\lambda}$
satisfy 
\begin{equation}\label{intervol}
B_2\left(\bm{0}, \rho
\right)\cap\left(C_{\bm{\lambda}}+\bm{x} \right)
\neq \varnothing. 
\end{equation}

Note that $\bm{x}$ only depends on $(\lambda_i)_{i \notin
  I_{\bm{U}}}$ and that $C_{\bm{\lambda}}$ depends only on $(\lambda_i)_{i \in
  I_{\bm{U}}}$. 

%
An immediate consequence of  Lemma~\ref{covariogramme} is that  the volume of
the intersection~\eqref{intervol} is less than the volume obtained
when setting $\bm{x}=\bm{0}$. In other words, for each fixed 
$\bm{x}$, 
\[
\begin{split}
& \Vol \left(\left\{ \left(
  \lambda_i\right)_{i\in I_{\bm{U}}}\in\R^{s-r} : \text{
  \eqref{ineqnormratiobis}  holds} \right\} \right) \\
\leq \ &  
\Vol \left(\left\{ \left(
  \lambda_i\right)_{i\in I_{\bm{U}}} : \left\|  P\left(\sum_{i\in
      I_{\bm{U}}} \lambda_i\left\|\bm{u}_i 
    \right\|_2 \cdot  \bm{e}_i \right)\right\|_2\; < \; \rho \right
\}
\right).
\end{split}
\]
%

From Lemma~\ref{lemindiceproj} and the choice of the index set
$I_{\bm{U}}$ we find that 
\begin{equation} \label{eqellipsoid}
\left\| \sum_{i\in I_{\bm{U}}} \lambda_i\left\|\bm{u}_i \right\|_2
  \cdot  \bm{e}_i \right\|_2\; < \; \kappa_s\rho 
\end{equation}
for some constant $\kappa_s>0$ depending only on
$s$. From~\eqref{defrho}, the measure of the ellipsoid determined
by~\eqref{eqellipsoid} is, up to a multiplicative constant depending
on the parameters $s, d, r$ and $N$,  
\begin{equation} \label{upperboundmeasestimate}
\prod_{i\in I_{\bm{U}}} \frac{\Phi(T)}{\left\|\bm{u}_i \right\|_2}\;
=\; \frac{\Phi(T)^{s-r}}{\prod_{i\in
    I_{\bm{U}}}\left\|\bm{u}_i\right\|_2}\cdotp 
\end{equation}
This upper bound 
is independent of the remaining
$r$ coordinates $(\lambda_i)_{i\not \in I_{\bm{U}}}$. When
integrating this bound against these $r$ coordinates when they vary within the range
\eqref{sizelabmda8i}, one obtains that the measure of the set of
vectors $\bm{\lambda} \in
(-\sqrt{d}N, \sqrt{d}N)^{s}$ such that~\eqref{defrho} holds
for a \emph{fixed} integer vector $\bm{p}$ is, up to another multiplicative
constant depending on $s, d, r$ and $N$, again bounded above
by~\eqref{upperboundmeasestimate}. 

Note that from ~\eqref{conditiontheta},
there are at most $\left(16\sqrt{d}N\right)^{s}\times
\prod_{i=1}^{s}\left\|\bm{u}_i\right\|_2$ vectors $\bm{p}$ to be
taken into account. Also, from the definition of the set
$\mathcal{V}_{s,d}(T)$ in~\eqref{defvsdT}, the inequality
$\left\|\bm{u}_i\right\|_2\le T$ holds for all $i\in \{1, \ldots, s
\}$. The measure of the set of vectors $\bm{\lambda} \in (-\sqrt{d}N,
\sqrt{d}N)^{s}$ such 
that~\eqref{defrho} holds for \emph{some} integer vector
$\bm{p}$ is thus, up to a multiplicative constant depending on $s, d,
r$ and $N$, at most $$\Phi(T)^{s-r}\cdot\prod_{i\not\in
  I_{\bm{U}}}\left\|\bm{u}_i\right\|_2\;\le\; \Phi(T)^{s-r}\cdot
T^r.$$ 
The lemma then follows upon integrating this bound according to the
decomposition~\eqref{dtheta} taking into account the right-hand side
of \eqref{sizelabmda8i}. 
\end{proof}

%

\begin{proof}[Completion of the proof of
  Theorem~\ref{thmforestperesbis}]
Under the assumptions of Theorem~\ref{thmforestperesbis}, using Proposition~\ref{propmultilinalg}, it is enough to prove that for
almost all matrices $\bm{\Theta}\in \R^{d\times s}$, there
are only finitely many values of $T\ge 1$ such that the
relation~\eqref{fundinegmultiudt}  holds for some matrix $\bm{U}\in
\mathcal{V}_{s,d}(T)$ and some vector $\bm{p}=\left(p_1, \dots,
  p_s\right)^T\in\Z^{s}$ satisfying~\equ{eq: 5.36}.
%
%

Given an integer $m\ge 1$ such that $2^m\le T<2^{m+1}$, it
follows from the monotonicity of the function $\Phi$ and from
assumption~\eqref{growthconstraint} that $$\Phi\left(T\right)\; \le \;
\Phi\left(2^m\right)\; \le \; \kappa\Phi\left(2^{m+1}\right)$$ for
some $\kappa>0$ and for all $m$ large enough. Since, clearly,
$\mathcal{V}_{s,d}(T)\subset \mathcal{V}_{s,d}\left(2^{m+1}\right)$,
this shows that it suffices to consider the case that $T$ is a
power of 2. 

Fix an integer $N\ge 1$. Then we see that 
Theorem~\ref{thmforestperesbis} is
implied by 
\begin{equation}\label{but2}
\Vol \left( \limsup_{m\rightarrow \infty} \left(\bigcup_{\bm{U}\in
      \mathcal{V}_{s,d}\left(2^m\right)}
    E_{s,d}^{(N)}\left(\bm{U},2^m\right)\right)\right) \; = \; 0.
\end{equation}
To establish this, 
decompose $
\mathcal{V}_{s,d}\left(T\right)$ as the disjoint
union $$\mathcal{V}_{s,d}\left(T\right)\; = \;
\bigcup_{r=1}^{d}\mathcal{V}_{s,d}^{(r)}\left(T\right),$$ where
$\mathcal{V}_{s,d}^{(r)}\left(T\right)$  denotes the set of matrices
in $\mathcal{V}_{s,d}\left(T\right)$ with rank $r$ and note that the
number of integral $s \times d$ matrices of norm at most $T$ is
$O(T^{sd})$. 
%
%
Thus
\begin{align*}
 \Vol\left( \bigcup_{\bm{U}\in \mathcal{V}_{s,d}\left(2^m\right)}
  E_{s,d}^{(N)}\left(\bm{U},2^m\right)\right) \; & \le \;
                                                            \sum_{r=1}^{d}
                                                            \sum_{\bm{U}\in
                                                            \mathcal{V}^{(r)}_{s,d}\left(2^m\right)}\Vol
                                                   \left(
                                                            E_{s,d}^{(N)}\left(\bm{U},2^m\right)\right)
  \\ 
 & \underset{(Lemma~\ref{lememasure})}{\ll} \;  2^{msd} \sum_{r=1}^{d}2^{m r}\cdot \Phi\left(2^m \right)^{s-r}\\
 & \ll \;  2^{m(s+1)d}\cdot \Phi\left(2^m \right)^{s-d},
\end{align*}
where the last relation follows from the trivial bound
$\Phi(T)\le 1$ guaranteed by~\eqref{lowboundudt}. Now
\eqref{convergencecondi} in conjunction with 
the Borel--Cantelli Lemma
(see e.g. \cite[Lemma C.1]{bugeaud})
imply~\eqref{but2}.
\end{proof}

\section{Some open questions}
In this section we collect some questions left open by our
discussion. 

\begin{enumerate}
\item
What is the actual optimal bound on the visibility function in Peres'
original example? 
Note that Peres gave a bound of $O(\vre^{-4})$, which we improved to
$O(\vre^{-3})$, but it is possible that this bound is also not tight. 
More generally, can one improve the visibility bounds of the 
sets $\mathfrak{F}\left(\bm{\Theta}_{s,d} \right)$ for appropriate choices of
$\bm{\Theta}_{s,d}$? 
Similarly, can one prove better visibility bound for the uniformly
discrete dense forest discussed in Theorem \ref{thm: finitely many
  lattices}?\\
\item
  As was pointed out by the referee, the discovery of the intriguing physical
  properties of twisted
  bilayer graphene (see \cite{BM, CB} and references therein) motivates
  the particular study of unions of translated lattices of the
  following form. Let
  $$\Lambda = \mathrm{span}_{\Z} \left(\left( \begin{matrix} 1 \\0 \end{matrix}
\right), \left ( \begin{matrix} \frac{1}{2} \\
    \frac{\sqrt{3}}{2} \end{matrix} \right) \right) \subset \R^2$$
 (the honeycomb lattice), let $\theta_1, \ldots, \theta_k \in
 \mathbb{S}^1$, let $r_{\theta_i} : \R^2 \to \R^2$ be the corresponding rotation
 matrices, and let $\bm{x}_1, \ldots, \bm{x}_k\in \R^2.$
 With this data, consider the union
 $$
\mathbf{TBG}\left(k, \left(\bm{x_i}\right), \left(\theta_i\right)\right) \df \bigcup_{i=1}^k\left(
r_{\theta_i}(\Lambda) + \bm{x}_i\right). 
 $$
 \begin{itemize}
   \item What is the smallest $k$ for which one can 
 find $\theta_i$ and $\bm{x}_i$ so that
 $
\mathbf{TBG}\left(k, \left(\bm{x_i} \right), \left(\theta_i\right)\right) 
$
is a dense forest? In particular, can one can take $k=2$? Can one
obtain uniformly discrete dense forests with $k=3$ (note that  by
Corollary \ref{cor: 2 grids not enough}, $k=2$ is
impossible). 
\item Is there $k$ for which $
\mathbf{TBG}\left(k, \left(\bm{x_i}\right), \left(\theta_i\right)\right) 
$ is a dense forest, for a.e. choice of $\theta_i$ and $\bm{x}_i$?
\item
  What visibility bounds can be obtained for fixed
  $k$ and for large $k$?
\end{itemize}
It is likely that all of these questions can be fruitfully
studied by adapting the techniques of this paper. \\
\item
For appropriate choices of the subspace $V$, give visibility bounds
for the uniformly discrete example of Theorem \ref{thm:UDDF_as_a_union_of_CP}.\\

\item
What is the best
  rate that the function $\Phi$ can attain for the set
  $UDT_{s}^{d}\left( \Phi\right)$ to be nonempty? \\

\item Explicit examples of badly approximable numbers / vectors / matrices
  are known: they are constructed from sets of algebraic
  conjugates. Can one find explicit examples of elements in
  $UDT_{s}^{d}\left( \Phi\right)$? \\

\item The notion of uniformly Diophantine set of vectors has not been considered
  before, but well-studied questions of Diophantine approximation are
  of interest here. For example, the Hausdorff dimension of
  $UDT_s^d(\Phi)$ for various choices of $\Phi$. Also, for which 
  choices of $\Phi$ does $UDT_s^d(\Phi)$ intersect nondegenerate
  analytic manifolds nontrivially? Note that besides its explicit
  interest, this is likely to be relevant to Question 1 above, as the
  conditions under which a union of lattices is uniformly discrete
  leads to the consideration of submanifolds in the space of
  lattices; see conditions (i) and (ii) of \S \ref{subsec:
  three lattice construction}. 

\end{enumerate}



\begin{thebibliography}{99}
\bibitem{Adiceam} F. Adiceam, \emph{How far can you see in a forest?},
  Int. Math. Res. Not. {\bf 16}, 4867--4881 (2016). 


\bibitem{Alon} N. Alon, \emph{Uniformly discrete forests with poor
    visibility}, Combin. Probab. Comput. 27 {\bf 4}, 442–-448 (2018).  


\bibitem{BG} M. Baake, U. Grimm, \underline{Aperiodic order. Vol 1: a
    mathematical invitation}, Cambridge University Press (2013).   


\bibitem{BW}  R. P. Bambah, A. C. Woods, {\em On a problem of
    Danzer}, Pacific J. Math. 37 {\bf 2}, 295--301 (1971). 


\bibitem{bana} W. Banaszczyk, \emph{New bounds in some transference
    theorems in the geometry of numbers.} Math. Ann. 296
  {\bf 4}, 625-–635  (1993).  


\bibitem{Bishop} C. Bishop, \emph{A set containing rectifiable arcs
    QC-locally but not QC-globally}, Pure Appl. Math. Q. {\bf 7} (1),
  121--138 (2011). 


\bibitem{BM} R. Bistritzer and A. H.  MacDonald, {\em Moir\'e bands
      in twisted double-layer graphene}. Proceedings of the National
    Academy of Sciences. {\bf 108} (30): 12233-–12237 (2011).
    
\bibitem{bourbaki} N. Bourbaki, \underline{Algebra 1},
  Springer--Verlag~: New-York (1989).  

\bibitem{bugeaud} Y. Bugeaud,  \underline{Distribution Modulo one and Diophantine Approximation}, Cambridge Tracts in Mathematics. No.~193. Cambridge~: At the University Press (2012).

\bibitem{bugeaudlaurent} Y. Bugeaud, L. Michel, \emph{On transfer
    inequalities in Diophantine approximation. II} 
Math. Z. 265 {\bf 2}, 249-–262 (2010). 

\bibitem{CB} Y. Cao, V. Fatemi, S. Fang, K. Watanabe, Y.  Taniguchi,
  E. Kaxiras and P. Jarillo-Herrero, {\em Unconventional
    superconductivity in magic-angle graphene
    superlattices}, Nature. 556: 43–-50 (2018).
  
\bibitem{Casselsbis} J.W.S. Cassels, \underline{An introduction to
    Diophantine approximation}, Cambridge Tracts in Mathematics and
  Mathematical Physics. No.~45. Cambridge~: At the University Press,
  x, 166 (1957). 

\bibitem{Cassels} J.W.S. Cassels, \underline{An Introduction to the
    Geometry of Numbers}, Berlin, Heidelberg, New-York: Springer
  (1971).  


\bibitem{CFS} I. P. Cornfeld, S. V. Fomin, Ya. G. Sinai,
  \underline{Ergodic theory}, Springer (1982). 




				







\bibitem{gruber} P. M. Gruber, \underline{Convex and discrete
    geometry}. Grundlehren der Mathematischen Wissenschaften
  [Fundamental Principles of Mathematical Sciences], 336. Springer,
  Berlin (2007). 


\bibitem{HKW} A. Haynes, M. Kelly, B. Weiss, \emph{Equivalence
    relations on separated nets arising from linear toral flows},
  Proc. Lond. Math. Soc. 109 {\bf 3}, 1203--1228  (2014). 



\bibitem{laurent} M. Laurent, \emph{On transfer inequalities in
    Diophantine approximation}, in \underline{Analytic number theory},
  306-–314, Cambridge Univ. Press, Cambridge (2009).  




\bibitem{miaoadi} J. Miao, A. Ben--Israel,  \emph{Product cosines of angles between subspaces}, Special issue honoring Calyampudi Radhakrishna Rao.
Linear Algebra Appl., no. 237/238, 71–81 (1996).







\bibitem{SchmidtDioph} W.M. Schmidt,  \underline{Diophantine
    approximation.} Lecture Notes in Mathematics, 785. Springer,
  Berlin (1980). 



\bibitem{schneider} R. Schneider,  \underline{Convex bodies: the
    Brunn-Minkowski theory.} Second expanded edition. Encyclopedia of
  Mathematics and its Applications, 151. Cambridge University Press,
  Cambridge (2014). 




\bibitem{SW} Y. Solomon, B. Weiss, \emph{Dense forests and Danzer
    sets},  Annal. Sci. de l'Ecole Norm. Super. {\bf 49} 1049--1070 (2016). 
    
\bibitem{whitney} H. Whitney,  \underline{Geometric integration theory.}  Princeton University Press, Princeton, N. J. (1957). 


\end{thebibliography}
\end{document}